\documentclass[a4paper, 12pt]{amsart}
\usepackage[a4paper, margin=1.7cm]{geometry}
\usepackage{amsmath,amsfonts,amssymb,amsthm}
\usepackage{mathrsfs}
\usepackage{mathtools}
\usepackage{graphics}
\usepackage[table]{xcolor}
\usepackage{epsfig}
\usepackage{esint}
\usepackage[pagebackref=true]{hyperref}
\usepackage{enumerate,enumitem}
\renewcommand\labelenumi{(\roman{enumi})}
\renewcommand\theenumi\labelenumi
\usepackage{multirow}

\newcommand{\ignore}[1]{}

\newtheorem{definition}{Definition}

\newtheorem{theorem}{Theorem}

\newtheorem{remark}{Remark}
\newtheorem{lemma}{Lemma}

\newtheorem{example}{Example}

\DeclareMathOperator{\supp}{supp}

\newcommand{\C}{\mathbb{C}}

\newcommand{\N}{\mathbb{N}}

\newcommand{\R}{\mathbb{R}}

\newcommand{\Z}{\mathbb{Z}}
\newcommand{\eps}{\varepsilon}
\newcommand{\id}{\mathsf{id}}

\DeclareFontFamily{U}{mathx}{}
\DeclareFontShape{U}{mathx}{m}{n}{<-> mathx10}{}
\DeclareSymbolFont{mathx}{U}{mathx}{m}{n}
\DeclareMathAccent{\widecheck}{0}{mathx}{"71}

\hypersetup{
pdfauthor={Jonas Sauer, Scott Smith},
pdftitle={Schauder estimates for germs by scaling},
breaklinks=true,
colorlinks=true,
linkcolor=blue,
citecolor=blue,
urlcolor=blue,
filecolor=blue,
}
\title[Schauder estimates for germs by scaling]{Schauder estimates for germs by scaling} 
\author{Jonas~Sauer}
\address{Jonas~Sauer\newline Friedrich-Schiller-Universit\"at Jena \newline Ernst-Abbe-Platz 2, 07737 Jena, Germany}
\email{jonas.sauer@uni-jena.de}
\author{Scott~Andrew~Smith}
\address{Scott~Andrew~Smith\newline Academy of Mathematics and System Sciences, Chinese Academy of Sciences \newline No. 55 Zhongguancun East Road, Beijing, China 100190 }
\email{ssmith@amss.ac.cn}

\begin{document}

\begin{abstract}
In this expository note, we show that the blow-up arguments of L. Simon adapt well to the corresponding Schauder theory of germs used in the study of singular SPDEs.  We illustrate this through some representative examples.  As in the classical PDE framework, the argument relies only on the scaling properties of the  germ semi-norms and the Liouville principle for the  operator.
\end{abstract}

\dedicatory{(In memory of Giuseppe Da Prato.)}

\maketitle
\tableofcontents

\section{Introduction}

Recent work on pathwise approaches to singular parabolic SPDEs have two central analytic ingredients, known as `reconstruction' and `integration' in the terminology of rough paths, regularity structures, and paracontrolled calculus (\cite{friz2020course}, \cite{GIP15}, \cite{Hai14}).  These tools can be viewed as methods for estimating the nonlinearity in the SPDE and the solution itself, respectively.  For example, in pathwise approaches to the multiplicative stochastic heat equation
\begin{equation}
(\partial_{t}-\Delta)u = \sigma(u)\xi \label{SPDE},  
\end{equation}
cf. \cite{HaP15},  \cite{OtW19}, \cite{chandra2024priori}, the reconstruction theorem is used to estimate the product $\sigma(u)\xi$, while the integration theorem exploits the regularizing properties of $(\partial_{t}-\Delta)^{-1}$ to estimate $u$.  In fact, typically the strategy is not to estimate either of these quantities directly, but rather the remainder that arises from re-centering them around one or more explicit functions (or distributions) which are amenable to direct calculations, such as multi-linear functionals of the noise $\xi$.  For an equation with additive noise like
\begin{equation}
(\partial_{t}-\Delta)u=-u^{3}+\xi,  
\end{equation}
the Da-Prato/Debussche \cite{da2003strong} trick amounts to considering a remainder of the form $u-v$, where $(\partial_{t}-\Delta)v=\xi$.  Even with $\xi$ being a space-time white noise in dimension 1, this yields a nice insight that despite both $u$ and $v$ individually belonging to $C^{1/2-}$, the difference $u-v$ belongs to $C^{5/2-}$, both H\"{o}lder spaces being understood with respect to the parabolic metric.  It is not immediately apparent what the analogous statement would be for \eqref{SPDE}, as a simple subtraction of $v$ amounts only to a change of the non-linearity to $(\sigma(u)-1)\xi$, which does not generally provide a cancellation.  The theory of regularity structures resolved this question as well as a closely related one, the  generalization of the classical Wong-Zakai theorem from ODE's to \eqref{SPDE}, cf. \cite{HaP15}.

\medskip

Writing down a single explicit equation for the remainder is often not possible or natural for many SPDEs of interest.  The basic idea in the rough path/regularity structures approach to singular SPDEs  is to instead consider a remainder depending on a base-point, so there is actually a `family of remainders' each satisfying a slightly different equation.  For example, freezing in the diffusion coefficient $\sigma(u)$ near a space-time point $z=(x,s)$ leads to 
\begin{equation}
(\partial_{t}-\Delta) \big(u-\sigma(u)(z)v-P_{z} \big)=\big (\sigma(u)-\sigma(u)(z)  \big )\xi \nonumber,   
\end{equation}
where $P_{z}$ is a polynomial in the kernel of $(\partial_{t}-\Delta)$.
One can then try to analyze the family of functions $(U_{z})_{z}$, where $U_{z}:=u-\sigma(u)(z)v-P_{z}$, which we refer to in the present paper as a germ, borrowing the terminology of \cite{broux2024hairer,CaZ20}.
A typical goal is to show that despite $u$ and $v$ being individually only $\alpha$-H\"{o}lder continuous, for a suitable choice of $P_{z}$, it holds that $U_{z}(w)$ vanishes to order $\eta>\alpha$ as the parabolic distance between $z$ and $w$ approaches zero.
This is accomplished with the Schauder estimate for germs.

\medskip

The `integration' theorem of \cite{otto2018parabolic}, inspired by the integration theorem for modelled distributions \cite{Hai14} and proved based on the Safonov approach as in \cite{OtW19}, is a generalized Schauder estimate which applies naturally to such germs.
In the present work, we show that various Schauder type estimates similar to the ones obtained in \cite{otto2018parabolic} and \cite{OSSW25} can alternatively be proved by the indirect blow-up arguments of \cite{simon1997schauder}.
The approach of \cite{simon1997schauder} is known to be rather useful and robust in classical PDE theory.
We note in advance that Hairer's seminal work \cite{Hai14} already cites \cite{simon1997schauder}, emphasizing in particular the importance of scaling, then uses a direct argument based on a kernel decomposition to prove a very general Schauder estimate for modelled distributions.
A closely related Schauder estimate for germs was also proved in \cite{broux2024hairer}.
Finally, we mention another kernel based approach used in \cite{LOTT21}. 
The goal of the present work is rather to illustrate that following more closely the indirect argument  of \cite{simon1997schauder} yields another accessible
proof\footnote{Of course with the drawback that indirect arguments lead to non-explicit constants in the estimates.} of some related Schauder estimates for germs, which appears to be fairly robust based on the examples we consider below.  As some further evidence of the robustness of Simon's method applied to germs, we mention the recent work \cite{esquivel2024prioriboundsdynamicfractional} which considers localized Schauder estimates in the context of the fractional heat equation.\footnote{The present work is a minor revision of a version submitted for publication on October 1, 2024, prior to the appearance of \cite{esquivel2024prioriboundsdynamicfractional} on ArXiv.}

\medskip

\subsection{Notation}
For a natural number $n \in \N$, we write $[n]$ for the integers from $1$ to $n$.  The space $C^0(\R^d)$ denotes all bounded and continuous functions $u:\R^d\to\C$.
For $k\in\N$ we introduce the spaces $C^k(\R^d)$ of all $u\in C^0(\R^d)$ such that $\partial^\gamma u\in C^0(\R^d)$ for all $|\gamma|\le k$.  Note that in general, the definition of $|\gamma|$ depends on the scaling $\mathfrak{s}$ introduced in the following section, cf. \eqref{e51}.

\section{An Illustration of the Method}\label{sec:illustration}
In this section, we illustrate the main idea of how to use Simon's \cite{simon1997schauder} blow-up arguments for elliptic/parabolic regularity of  germs.  The basic setup is that we have an ambient space $\R^{d}$ and a subset $D \subseteq \R^{d}$.  The most relevant examples to keep in mind are when $D$ is simply an open subset of $\R^{d}$ or a lattice.  We fix a grading of each of the $d$ variables, namely $\mathfrak{s}=(\mathfrak{s}_{1},\cdots, \mathfrak{s}_{d}) \in \N^{d}$, which determines a notion of degree: namely, for a multi-index $\gamma\in\N_0^d$ we introduce
\begin{align}
|\gamma|:=\mathfrak s \cdot \gamma = \sum_{i=1}^d \mathfrak{s}_i \gamma_i. \label{e51}
\end{align}
We are interested in the Schauder theory of an elliptic operator with constant (complex) coefficients of the form 
\begin{equation}
 \mathcal{L}=\sum_{|\gamma|=m} a_\gamma \partial^\gamma \label{e52},   
\end{equation}
with $m\in\N$.  In the present context, ellipticity means that for every domain $D\subseteq \R^d$ and distribution $u\in \mathcal{D}'(D)$ one has $u\in C^\infty(D)$ if $\mathcal{L}u\in C^\infty(D)$, cf. Section 2 in \cite{Tre75}.\footnote{Some authors use the terminology hypoelliptic, but since we only consider scaling homogeneous operators with constant coefficients, we use the term elliptic to avoid confusion.}  For most of the present section, we have in mind second order operators $\mathcal{L}$, the simplest examples being the Laplacian and the heat operator (which is of order $2$ in an appropriate anisotropic framework), the treatment of general  operators is postponed to Section \ref{sec:generalOperators}.

\medskip

As alluded to in the introduction, for applications to singular SPDE one requires a Schauder theory not only for functions, but for germs, which we now define.
\begin{definition}[Germs] \label{def:germs}
    A germ $U$ over a subset $D \subseteq \R^{d}$ is a family $(U_{x})_{x \in D}$ of continuous functions $y \in D \mapsto U_{x}(y) \in \C$.
\end{definition}
To measure germs, we use a generalization of the classical H\"{o}lder semi-norms.  Recall that the scaling $\mathfrak{s}$ determines a natural anisotropic distance\footnote{To lighten notation, in comparison to \cite{Hai13}, we drop the sub-script $\mathfrak{s}$.} defined for $x,y \in \R^{d}$ by
\begin{equation}
d(x,y):=\sum_{i=1}^{d}|x_{i}-y_{i}|^{\frac{1}{\mathfrak{s}_{i}}} \label{e35}.
\end{equation}
Throughout the paper, given a base-point $x \in \R^{d}$ and a radius $R>0$, we denote by $B_{R}(x)$ the anisotropic ball with respect to $d$.

\medskip

The basic goal of Schauder theory for germs is to estimate the following $G^{\eta}(D)$ norm, denoted by $\|U\|_{G^{\eta}(D)}$ and defined to be the infimum over all  $M>0$ such that for all $x,y \in D$
\begin{equation}
    |U_{x}(y)| \leq M d^{\eta}(x,y).\label{e30}
\end{equation}
The above norm appears already in a one-dimensional setting in the work of Gubinelli \cite{Gub04}, and is a basic quantity in the modern theory of rough paths, cf. \cite{friz2020course}.  For a particular choice of germ relevant to quasi-linear SPDE, it was used in a multi-dimensional setting in \cite{OtW19} and also plays an important role in the follow-up works \cite{otto2018parabolic} and \cite{OSSW25}, see also \cite{moinat2020space} for an application to the $\Phi^{4}_{3}$ model.
\begin{remark}
In general, if $D$ is a bounded domain in $\R^{d}$, one may wish to include an additional weight in  \eqref{e30} which degenerates at the boundary.  This is not necessary for the examples we have in mind in this section, since we work either in an unbounded domain or we assume the germ is well-behaved at the relevant boundary.
\end{remark}

Typically, we refer to $x \in D$ in the above definition as the `base point' and $y$ as the `active variable'.  
Classical Schauder estimates, cf. \cite{simon1997schauder}, control a semi-norm of a \textit{single function} $u$ by a suitable semi-norm of $\mathcal{L}u$, on the H\"{o}lder scale.  In the present setting, given a germ $U$ over an open subset $D \subseteq \R^{d}$, $\mathcal{L}U$ is understood to be a family of distributions $(\mathcal{L}U_{x})_{x \in D}$, that is $\mathcal{L}$ is understood to be applied to the active variable with a \textit{fixed} base-point.  In fact, $\mathcal{L}U$ is a distribution-valued germ in the following sense. 
\begin{definition}[Distribution-Valued Germs]
A distribution-valued germ $V$ over an open subset $D \subseteq \R^{d}$ is a family $(V_{x})_{x \in D}$ where each $V_{x} \in \mathcal{D}'(D)$.
\end{definition}
We now introduce the $G^{\gamma}(D)$ norm for $\gamma<0$.  Let $\mathcal{B}\subseteq\mathcal{D}(\R^d)$ denote the subset of functions $\varphi$ in the unit ball of $C^k(\R^d)$, $k:=\lceil|\gamma|\rceil$,\footnote{For simplicity, we don't indicate the specific integer $k$ in the definition of $\mathcal{B}$. We remark that the value of $k$ may always be increased without changing the $G^\gamma(D)$ norm, cf.\@ Proposition B.1 in \cite{broux2024hairer}.} with $\supp\varphi\subseteq B_1(0)$.
Given a base-point $x$ and a convolution scale $\lambda>0$, following \cite{Hai13} we define a re-scaled test function
\begin{equation}
    \varphi_{x}^{\lambda}(y):=\lambda^{-\sum_{i=1}^{d}\mathfrak{s}_{i} }\varphi \big (\lambda^{-\mathfrak{s}_{1}}(y_{1}-x_{1} ),\cdots,\lambda^{-\mathfrak{s}_{d}}(y_{d}-x_{d} )  \big ), \nonumber
\end{equation}
which is the typical anisotropic re-scaling designed to preserve the $L^1$-norm of $\varphi$.  For $\gamma<0$ the $G^{\gamma}(D)$ semi-norm is defined by
\begin{equation}
    [V]_{G^{\gamma}(D)}:= \sup_{\varphi\in \mathcal{B}}\sup_{x \in D}\sup_{\substack{\lambda>0 \\ B_{\lambda}(x)\subseteq D}}\lambda^{-\gamma} \big | \langle V_{x},\varphi_{x}^{\lambda}  \rangle \big | \label{e34}
\end{equation}
Note that the constraint $B_{\lambda}(x) \subseteq D$ ensures that for $\varphi \in \mathcal{B}$, it holds $\varphi_{x}^{\lambda} \in C^{\infty}_{c}(D)$ so that the pairing $\langle V_{x},\varphi_{x}^{\lambda}  \rangle$ is well-defined for $V_{x} \in \mathcal{D}'(D)$.
In the Schauder theory of germs, controlling \eqref{e30} requires not only a bound on $\mathcal{L}U$, but also on an additional semi-norm of $U$, which we now introduce.
For an integer $k$, we let $\mathcal{P}_{k}(\R^{d})$ denote the polynomials of anisotropic degree (according to \eqref{e51}) of order $k$.  For $0<\alpha<\eta$, we define the $G^{\eta,\alpha}(D)$ semi-norm, denoted $[U]_{G^{\eta,\alpha}(D)}$, to be the infimum over all $M>0$ such that for all $x,y \in D$ there exists a polynomial $P\in\mathcal{P}_{\lfloor\eta\rfloor}$ such that for all $z \in D$ 
\begin{equation}
    \big |(U_{x}-U_{y}-P)(z)  \big | \leq M {d(y,z)^{\alpha}\big(d(x,y)+ d(y,z)\big )^{\eta-\alpha}} \label{e41}.
\end{equation}
Like \eqref{e30}, this semi-norm is closely related to a semi-norm on functions depending on three arguments in the work of Gubinelli \cite{Gub04} in the one dimensional setting, where the polynomial there is of degree zero.  The reader might be curious why the above semi-norm does not appear in classical Schauder theory: see Remark \ref{remark:classical} for more details on this point. 

\medskip

We now comment on the scaling properties of the above semi-norms, as this is an important ingredient of the method of Simon \cite{simon1997schauder}.  To this end, we introduce some further notation: for $w\in \R^d$ and $R>0$ we define the rescaling/recentering map $S^{R}_{w}$ by
\begin{align*}
    y \mapsto S_{w}^{R}y:=w+R^{\mathfrak{s}} y,
\end{align*}
where $R^{\mathfrak{s}} y:=(R^{\mathfrak{s}_1} y_1,\ldots, R^{\mathfrak{s}_d} y_d)$.  This transformation has a natural action on germs.  Indeed, given a germ $U$ over $D$, we denote by $S^{R}_{w}U$ the germ over $(S^{R}_{w})^{-1}D$ defined by 
\begin{equation}
    \big ( S^{R}_{w}U \big )_{x}:=U_{S^{R}_{w}x} \circ S^{R}_{w} \label{e38}.
\end{equation}
Observe that in terms of the scaling transformations, we can also write $\varphi_x^\lambda=\lambda^{-\sum_{i=1}^{d}\mathfrak{s}_{i}}\varphi\circ (S_{x}^{\lambda})^{-1}$.  Thus, for a distributional germ $V$ over an open subset $D$, the change of variables formula suggests to define a distributional germ $S^{R}_{w}V$ over $(S^{R}_{w})^{-1}D$ via
\begin{equation*}
    \langle \big ( S^{R}_{w}V \big )_{x}, \varphi \rangle:=  \langle V_{S^{R}_{w}x}, \varphi^{R}_{w}  \rangle \nonumber.
\end{equation*}
The following lemma is then an elementary consequence of change of variables.
\begin{lemma}\label{lemma:Rescaling}
Let $R>0$ and $w \in \R^{d}$. For a subset $D \subseteq \R^{d}$ and $0<\alpha<\eta$ the following identities hold for all germs $U$ over $D$
\begin{equation}
    \|S^{R}_{w}U\|_{G^{\eta}((S^{R}_{w})^{-1}D)}=R^{\eta} \|U\|_{{G^{\eta}(D)}}, \quad [S^{R}_{w}U]_{G^{\eta,\alpha}((S^{R}_{w})^{-1}D)}=R^{\eta} [U]_{{G^{\eta,\alpha}(D)}} \label{e39}.
\end{equation}
Furthermore, if $D$ is open and $\gamma<0$, then for all distributional germs $V$ over $D$ it holds
\begin{equation}
    [S^{R}_{w}V]_{G^{\gamma}((S^{R}_{w})^{-1}D)}=R^{\gamma} [V]_{{G^{\gamma}(D)}}.\label{e40}
\end{equation}
\end{lemma}
\begin{proof}
For the first identity, observe that for any $x,y \in \R^{d}$ it holds $d(S^{R}_{w}x,S^{R}_{w}y)=Rd(x,y)$, so the claim follows from observing that $S^{R}_{w}$ is a bijection from $(S^{R}_{w})^{-1}D$ to $D$.  The second identity uses the previous observation, together with the bijection from $\mathcal{P}_{\lfloor \eta \rfloor}$ to $\mathcal{P}_{\lfloor \eta \rfloor}$ given by $P \mapsto P \circ (S_{w}^{R})^{-1}$.
For the third identity, we first note the following: for any $\lambda, \rho >0$ and $x,z \in \R^{d}$ it holds
\begin{equation}
S_{x}^{\lambda}\circ S_{z}^{\rho}=S_{S_x^\lambda z}^{\rho\lambda}, \qquad (S_{x}^{\lambda})^{-1}=S_{-(\lambda^{-1})^{\mathfrak{s}}x}^{\lambda^{-1}} \label{e48}.
\end{equation}
Hence, we find that
\begin{equation}
    \langle (S^{R}_{w}V)_{x}, \varphi_{x}^{\lambda} \rangle=\langle V_{S^{R}_{w}x},(\varphi^{\lambda}_{x})^{R}_{w} \rangle=\langle V_{S^{R}_{w}x},\varphi^{R\lambda }_{S_w^Rx} \rangle \label{e49},
\end{equation}
where we used \eqref{e48} to obtain
\begin{equation}
(\varphi^{\lambda}_{x})^{R}_{w}=(\lambda R)^{-\sum_{i=1}^{d}\mathfrak{s}_{i}}\varphi \circ (S^{\lambda}_{x})^{-1} \circ (S^{R}_{w})^{-1}=(\lambda R)^{-\sum_{i=1}^{d}\mathfrak{s}_{i}}\varphi \circ (S^{R}_{w} \circ S^{\lambda}_{x})^{-1}=\varphi^{R\lambda }_{S_w^Rx} \nonumber.    
\end{equation}
The third identity now follows directly from \eqref{e49} taking into account that $B_{\lambda}(x) \subseteq (S^{R}_{w})^{-1}D$ if and only if $B_{R\lambda} (S_w^Rx) \subseteq D$.
\end{proof}
For future reference, we note that \eqref{e40} is particularly useful in the form
\begin{equation}
[\mathcal{L}S^{R}_{w}U]_{G^{\gamma}((S^{R}_{w})^{-1}D)}=R^{\gamma+m} [\mathcal{L}U]_{{G^{\gamma}(D)}}\label{e53},
\end{equation}
which follows from noting that $\mathcal{L}S^{R}_{w}=R^{m}S^{R}_{w}\mathcal{L}$.

\medskip

\subsection{From Germ Bounds to H\"{o}lder Bounds} 
In typical applications to singular SPDE, see for example \cite{otto2018parabolic} or \cite{OSSW25}, the parameter $\alpha$ in the $G^{\eta,\alpha}(D)$ semi-norm is known in advance to be the expected optimal H\"{o}lder regularity of the function $U_{x}$ for a fixed base-point $x$.
For instance, the germ based on freezing of coefficients of \eqref{SPDE} mentioned in the introduction would have H\"{o}lder regularity to order $\alpha<\frac{1}{2}$ for space-time white noise.
This however is generally not sufficient to apply the reconstruction theorem and obtain pathwise bounds on the (renormalized) non-linear term in \eqref{SPDE}.
The aim of Schauder theory for germs is to obtain a form of higher regularity $\eta>\alpha$, then retrieve traditional H\"{o}lder regularity as a corollary.
In fact, it turns out to be useful for the proof of the Schauder estimate to collect this corollary in advance: namely, for suitable subsets $D \subseteq \R^{d}$, if $\|U\|_{G^{\eta}(D)}+[U]_{G^{\eta,\alpha}(D)}<\infty$, then each $U_{x}$ is (locally) $\alpha$-H\"{o}lder continuous.   
\begin{lemma}\label{lemma:Holderbound2}
Let $0<\alpha<1<\eta<2$ and assume $D \subseteq \R^{d}$ is a subset with the following property: for each $i \in [d]$ with $\mathfrak{s}_{i}=1$, if $x,y \in D$ then $y+d(x,y)\mathbf{e}_{i} \in D$.
Then for all germs $U$ over $D$ and all $R>0$
 \begin{equation}
     \sup_{x \in D}[U_{x}]_{C^{\alpha}( D \cap B_{R}(x))} \lesssim \big ( \|U\|_{G^{\eta}(D)}+[U]_{G^{\eta,\alpha}(D)} \big ) R^{\eta-\alpha} \label{e33}.
 \end{equation}
\end{lemma}
\begin{remark}
The restriction $\eta \in (1,2)$ is made for simplicity of exposition in this section, as this is the typical case of interest for applications to second order sub-critical SPDEs. It will be removed in Section \ref{sec:further} where we consider a more general class of operators.
\end{remark}
\begin{proof}
Let $U$ be a germ with $\|U\|_{G^{\eta}(D)}+[U]_{G^{\eta,\alpha}(D)}<\infty$. Fix a base point $x \in D$, a radius $R>0$, and let $y,z \in D \cap B_{R}(x)$.  We would like to estimate the increment $U_{x}(z)-U_{x}(y)$, and we first observe that since $[U]_{G^{\eta,\alpha}(D)}<\infty$ and $\eta \in (1,2)$, there exists a polynomial $P_{x,y} \in \mathcal{P}_{1}$ such that
\begin{align}
\big |U_{x}(z)-U_{y}(z)-P_{x,y}(z) \big | 
&\leq 2[U]_{G^{\eta,\alpha}(D)}  d(y,z)^{\alpha}\big (d(y,z)+d(x,y) \big )^{\eta-\alpha} \label{e22} \\
& \lesssim [U]_{G^{\eta,\alpha}(D)}d(y,z)^{\alpha}R^{\eta-\alpha}
,\label{e12}
\end{align}
where we used the triangle inequality in the form $d(y,z)\leq 2R$ in the second step.  Since $\|U\|_{G^{\eta}(D)}<\infty$, we have  $|U_{y}(z)| \leq \|U\|_{G^{\eta}(D)}d(y,z)^{\eta} \lesssim \|U\|_{G^{\eta}(D)} d(y,z)^{\alpha}R^{\eta-\alpha}$.
By the triangle inequality, \eqref{e12} hence turns into
\begin{equation}
\big |U_{x}(z)-P_{x,y}(z) \big | 
\lesssim \big(  [U]_{G^{\eta,\alpha}(D)}+\|U\|_{G^{\eta}(D)}   \big )d(y,z)^{\alpha}R^{\eta-\alpha} .\label{e13}
\end{equation}
Note that $P_{x,y}(y)=U_x(y)$, so
\begin{equation}
 P_{x,y}=U_x(y)+\sum_{i \in [d], \mathfrak{s}_i=1}\gamma^{i}_{x}(y) (\cdot-y)_{i} \label{e36}   
\end{equation}
for some $\gamma^{i}_{x}(y) \in \R$.  We claim that for each $i \in [d]$ with $\mathfrak{s}_i=1$, 
\begin{equation}
\|\gamma^{i}\|_{G^{\eta-1}(D) }\lesssim   \|U\|_{G^{\eta}(D)}+[U]_{G^{\eta,\alpha}(D)} \label{e15}.
\end{equation} 
Combining this with $|P_{x,y}(z)-U_x(y)| \leq \|\gamma^{i}\|_{G^{\eta-1}(D)} d(x,y)^{\eta-1}d(y,z) \leq \|\gamma^{i}\|_{G^{\eta-1}(D)} R^{\eta-\alpha}d(y,z)^{\alpha}$ implies \eqref{e33}.

\medskip

We now give the argument for \eqref{e15} following Step 6 in Proposition 2 of  \cite{OSSW25}, which is modelled on the proof of Lemma 3.3 in \cite{OtW19}.  Fix $i \in [d]$ with $\mathfrak{s}_{i}=1$ and consider any $x,y \in D$. By assumption, the point $w:=y+\mathbf{e}_{i}d(x,y)$ belongs to $D$ and from the definition \eqref{e35}, we see that  $d(y,w)=d(x,y)$.  From \eqref{e22} (with $w$ in place of $z$), \eqref{e36} and the triangle inequality
\begin{align}
    |\gamma_{x}^{i}(y)| d(x,y) &\lesssim \big ( \|U\|_{G^{\eta}(D)}+[U]_{G^{\eta,\alpha}(D)} \big ) d(x,y)^{\eta}+\big |U_{x}(w)-U_{x}(y)-U_{y}(w) \big | \nonumber \\
    &\lesssim ( \|U\|_{G^{\eta}(D)}+[U]_{G^{\eta,\alpha}(D)} \big ) d(x,y)^{\eta},
\end{align}
where we used the definition \eqref{e30} and the triangle inequality in the last step, completing the argument for \eqref{e15} and hence the proof of the Lemma.
\end{proof}
We now turn to a related property of the germ $\mathcal{L}U$, where we assume $\mathcal{L}$ is given by \eqref{e52}.  Namely, if $D$ is an open subset and $\mathcal{L}U$ has a finite $G^{\eta-m}(D)$ semi-norm for $\eta<m$ and $U$ has a finite $G^{\eta,\alpha}$ semi-norm, then for each base-point $x$, the function $\mathcal{L}U_{x}$ is H\"{o}lder continuous of order $\alpha-m$.
\begin{lemma} \label{lemma:Holderbound3}
Let $\mathcal{L}$ be given by \eqref{e52} and $0<\alpha<\eta<m$ and $D \subseteq \R^{d}$ be open.  Then for all germs $U$ over $D$ and $R>0$
\begin{equation}
    \sup_{x \in D} \|\mathcal{L}U_{x}\|_{C^{\alpha-m}(D \cap B_{R}(x)) } \lesssim \big ( \|\mathcal{L}U\|_{G^{\eta-m}(D)}+[U]_{G^{\eta,\alpha}(D)} \big )R^{\eta-\alpha} \label{e37}
\end{equation}
\end{lemma}
\begin{proof}
Let $x,y \in D$ and $\lambda>0$ satisfy $B_{\lambda}(y) \subseteq B_{R}(x) \cap D$.  By the triangle inequality, for $\varphi \in \mathcal{B}$ it holds
\begin{align}
  \big | \langle  \mathcal{L}U_{x},\varphi_{y}^{\lambda} \rangle \big |  &\leq \big | \langle  \mathcal{L}U_{y},\varphi_{y}^{\lambda} \rangle \big |+\big | \langle  U_{y}-U_{x},\mathcal{L}^{*}\varphi_{y}^{\lambda} \rangle \big | \nonumber \\
  &\leq \|\mathcal{L}U\|_{G^{\eta-m}(D)}\lambda^{\eta-m}+ [U]_{G^{\eta,\alpha}(D)} \langle d(y,\cdot)^{\alpha} (d(x,y)+d(y,\cdot))^{\eta-\alpha},|\mathcal{L}^{*}\varphi^{\lambda}_{y}|  \rangle \nonumber \\
  & \lesssim \|\mathcal{L}U\|_{G^{\eta-m}(D)}\lambda^{\eta-m} +[U]_{G^{\eta,\alpha}(D)} \lambda^{\alpha-m}\big ( d(x,y)^{\eta-\alpha}+\lambda^{\eta-\alpha} \big ) \nonumber \\
  &\lesssim \big ( \|\mathcal{L}U\|_{G^{\eta-m}(D)}+[U]_{G^{\eta,\alpha}(D)} \big ) \lambda^{\alpha-m}R^{\eta-\alpha}, \nonumber 
  \end{align}
where we used that $\mathcal{L}P=0$ for any $P \in  \mathcal{P}_{m-1}$ together with the inequalities $\lambda \leq R$ and $\eta>\alpha$.
\end{proof}

\subsection{The Laplacian on $\R^{d}$} \label{sec:Laplacian}
We start with the Schauder estimate for $\Delta$ on $\R^{d}$, which is the simplest setting to illustrate the method.  Hence, in this subsection we choose the grading $\mathfrak{s}=(1,\cdots,1)$, which determines the relevant semi-norms cf. \eqref{e35}, \eqref{e30}, \eqref{e34} and \eqref{e41}.  We will prove the following Schauder estimate.
\begin{theorem}\label{thm:continuum}
Let $0<\alpha<1<\eta<2$.  There exists a constant $C>0$ such that for all germs $U$ with $\|U\|_{G^{\eta}(\R^{d})}<\infty$ it holds
\begin{equation}
    \|U\|_{G^{\eta}(\R^{d})} \leq C \big (  [\Delta U]_{G^{\eta-2}(\R^{d})}+[U]_{G^{\eta,\alpha}(\R^{d})} \big ) \label{e10}.
\end{equation}

\end{theorem}
The two main ingredients for Simon's \cite{simon1997schauder} method are \textit{scaling} and \textit{Liouville's theorem}.  The outline of the argument is as follows: since we can always multiply \eqref{e10} by a constant, an equivalent statement is: there exists a $C>0$ such for all germs with $\|U\|_{G^{\eta}(\R^{d})}=1$,
\begin{equation}
 [\Delta U]_{G^{\eta-2}(\R^{d})}+[U]_{G^{\eta,\alpha}(\R^{d})} \geq C^{-1}. \nonumber  
\end{equation}
Hence, if Theorem \ref{thm:continuum} is false, there exists a sequence of  germs $U^{n}$ with $\|U^{n}\|_{G^{\eta}(\R^{d})}=1$ for which $[\Delta U^{n}]_{G^{\eta-2}(\R^{d})}+[U^{n}]_{G^{\eta,\alpha}(\R^{d})}$ converges to zero.
The strategy is to use this information to construct a harmonic function which violates Liouville's theorem.
To pass from germs to functions, we observe first that $\|U^{n}\|_{G^{\eta}(\R^{d})}=1$ implies there is a base-point $x^{n}$ where a `concentration' occurs, namely there exists $y^{n}$ for which $(R^{n})^{-\eta}|U_{x^{n}}^n(y^{n})| \geq \frac{1}{2}$, where $R^{n}:=d(x^{n},y^{n})$.  We re-center and re-scale around $x^{n}$ by setting
\begin{equation}
    y \in \R^{d} \mapsto u^{n}(y):=(R^{n})^{-\eta}U^{n}_{x^{n}}(x^{n}+R^{n}y). \nonumber
\end{equation}
The aim is to show that (modulo a subsequence), $u^{n}$ converge locally uniformly to a harmonic function $u$ with sub-quadratic growth, namely $|u(y)| \leq d(0,y)^{\eta}$.  This will turn out to contradict Liouville's theorem as in the classical work \cite{simon1997schauder}.
\begin{proof}[Proof of Theorem \ref{thm:continuum}]
If the statement is false, there exists a sequence of germs $U^{n}$ on $\R^{d}$ such that $\|U^{n}\|_{G^{\eta}(\R^{d})}=1$ and $\|\Delta U^{n}\|_{G^{\eta-2}(\R^{d})}+[U^{n}]_{G^{\eta,\alpha}(\R^{d})} \leq \frac{1}{n}$. In particular, there exists $x^{n},y^{n} \in \R^{d}$ such that for $R^{n}:=d(x^{n},y^{n})$, it holds
\begin{equation}
    (R^{n})^{-\eta}|U_{x^{n}}^n(y^{n})| \geq \frac{1}{2} \label{e6}.
\end{equation}
Recalling \eqref{e38}, consider the re-centered/re-scaled germs $(R^{n})^{-\eta}S^{R_{n}}_{x^{n}}U^{n}$ and define a sequence of functions $u^{n}: \R^{d} \mapsto \R$ by choosing the base-point to be the origin.
That is, let $u^{n}:=(R^{n})^{-\eta} \big ( S^{R_{n}}_{x^{n}}U^{n} \big )_{0} $.  By scaling relations in Lemma \ref{lemma:Rescaling}, combined with Lemmas \ref{lemma:Holderbound2} and \ref{lemma:Holderbound3}, and \eqref{e53} we find for all $R>0$ 
\begin{equation}
    [u^{n}]_{C^{\alpha}(B_{R})} \lesssim R^{\eta-\alpha}, \qquad  \|\Delta u^{n}\|_{C^{\alpha-2}(B_{R})} \lesssim \frac{1}{n}R^{\eta}, \qquad \sup_{y \neq 0} d(y,0)^{-\eta}|u^{n}(y)| \leq 1 \label{e42}.
\end{equation}
The Arzel\`a-Ascoli theorem provides compactness along a subsequence (not re-labelled), along with locally uniform convergence to a continuous limit $u: \R^{d} \to \R$ with sub-quadratic growth
\begin{equation}
    \sup_{y \neq 0}d(0,y)^{-\eta} |u(y)| \leq 1. \label{e16}
\end{equation}
Hence, by the second estimate in \eqref{e42} it holds $\|\Delta u\|_{C^{\alpha-2}(B_{R})}=0$ for all $R>0$, which implies $u$ is harmonic.   
Due to the upper bound \eqref{e16} and  $\eta \in (1,2)$, it follows from Liouville's theorem that $u$ is a polynomial of degree at most one, but the only polynomial satisfying \eqref{e16} is identically zero.  On the other hand, \eqref{e6} implies $|u^{n}(z^{n})| \geq \frac{1}{2}$ for $z^{n}:=(S^{R^{n}}_{x^{n}})^{-1}y^{n}$.  Since $d(z^{n},0)=1$, the sequence $z^{n}$ has a subsequential limit $z\in \R^{d}$ with $d(z,0)=1$, and due to the the local uniform convergence it holds $|u(z)| \geq \frac{1}{2}$, a contradiction.
\end{proof}
\begin{remark}\label{remark:hypoelliptic}
It is not too difficult to see that the above argument is robust enough to replace $\Delta$ by any scaling homogeneous elliptic operator on $\R^{d}$ with constant coefficients and relax the restriction on $\eta$.  Indeed, the only place which requires an additional argument is the bound \eqref{e33} needed for compactness, which requires us to account for higher order polynomials in Lemma \ref{lemma:Holderbound2}.  This is carried out in Section \ref{sec:further} below, see Lemma \ref{lemma:Holderbound_gen_hyb}.
\end{remark}

\subsection{The Heat Operator on $(0,T) \times \R^{d}$}
In this section, we discuss how to adapt the ideas from Section \ref{sec:Laplacian} to analyze the initial value problem for the heat operator $\partial_{t}-\Delta$.
In line with Remark \ref{remark:hypoelliptic}, since this operator is elliptic, the corresponding full-space estimate \eqref{e10} follows from an identical argument if we replace $\Delta$ by $\partial_{t}-\Delta$, see Theorem \ref{thm:wholespace} below.
Thus, in the present section our aim is to  address the additional subtleties that arise from the presence of an initial time boundary.  We use the ambient space $\R^{d+1}$ in this sub-section with the scaling $\mathfrak{s}=(2,1,\dots,1)$  as well as the notation $x=(x_{0},x_{1}) \in \R \times \R^{d}$ to distinguish time and space.

\medskip

In comparison with the previous section, the Schauder estimate requires a control on the germ at the initial time $(0,x_{1})$.  We denote the initial germ by $U \mid_{x_{0}=0}$, which is a germ on $\R^{d}$ defined for a base-point $x_{1} \in \R^{d}$ and active variable $y_{1} \in \R^{d}$ by $U_{(0,x_{1})}(0,y_{1})$.  The inclusion of a term monitoring the size of the initial behaviour mirrors the classical Schauder estimate for the heat equation
\begin{align*}
    [u]_{C^{\alpha+2}((0,T)\times\R^d)}\le C\big([(\partial_t-\Delta)u]_{C^\alpha((0,T)\times\R^d)}+[u(0,\cdot)]_{C^{\alpha+2}(\R^d)}\big),
\end{align*}
see e.g.\@ Exercise 9.1.4 in \cite{krylov1996lectures}.
In the setting of germs, the estimate takes the following form.
\begin{theorem}\label{thm:IVP}
Let $0<\alpha<1<\eta<2$.  There exists a universal constant $C$ such that for all $T>0$ and germs $U$ over $[0,T) \times \R^{d}$ with $\|U\|_{G^{\eta}((0,T) \times \R^{d})}<\infty$ it holds
\begin{align*}
&\|U\|_{G^{\eta}( (0,T) \times \R^{d}) } \leq C \big (  [(\partial_{t}-\Delta)U]_{G^{\eta-2}((0,T)\times\R^{d}) }+[U]_{G^{\eta,\alpha}((0,T) \times \R^{d}) }+\|U \mid_{x_{0}=0} \|_{G^{\eta}(\R^{d})}\big ).
\end{align*}
\end{theorem} 
\begin{proof}
We use $\mathcal{L}$ to denote $\partial_{t}-\Delta$ throughout the proof.
If the statement is false, there exists a sequence of times $T^{n}>0$ and germs $U^{n}$ over $[0,T^{n}) \times \R^{d}$ such that $[U^{n}]_{G^{\eta}( (0,T^{n}) \times \R^{d}) }=1$ while $[\mathcal{L}U^{n}]_{G^{\eta-2}((0,T^{n})\times\R^{d}) }+\|U^{n}\|_{G^{\eta,\alpha}((0,T^{n}) \times \R^{d}) }+[U^{n} \mid_{x_{0}=0} ]_{G^{\eta}(\R^{d})} \leq \frac{1}{n}$.
In particular, there exist $x^{n},y^{n} \in (0,T^{n}) \times \R^{d}$ such that for $R^{n}:=d(x^{n},y^{n})$, it holds that $(R^{n})^{-\eta}|U_{x^{n}}^n(y^{n})| \geq \frac{1}{2}$.
We consider the re-centered/re-scaled germ $(R^{n})^{-\eta}S_{x^{n}}^{R^{n}}U^{n}$ over 
\begin{equation}
 D^{n}:=\big  ( S_{x^{n}}^{R^{n}} \big )^{-1} \big ( (0,T^{n}) \times \R^{d} \big )=(R^{n})^{-2}(-x^{n}_{0},T^{n}-x^{n}_{0} ) \times \R^{d} \nonumber,  
\end{equation}
and define $u^{n}:D^{n} \mapsto \R$ by $u^{n}:=(R^{n})^{-\eta} \big ( S_{x^{n}}^{R^{n}}U^{n} \big )_{0}$.  To obtain H\"{o}lder bounds, we rely again on Lemma \ref{lemma:Holderbound2}, noting in advance that $\mathfrak{s}_i=1$ only for $i>1$, so the desired property of $D$ is trivially satisfied in the spatial directions. Combining Lemmas \ref{lemma:Rescaling}, \ref{lemma:Holderbound2} and \ref{lemma:Holderbound3} and keeping in mind \eqref{e53}, we find that for all $R>0$ 
\begin{equation}
    [u^{n}]_{C^{\alpha}(B_{R} \cap D^{n})} \lesssim R^{\eta-\alpha}, \qquad  \|\mathcal{L} u^{n}\|_{C^{\alpha-2}(B_{R} \cap D^{n})} \lesssim \frac{1}{n}R^{\eta}, \qquad \sup_{y \in D^{n}, y \neq 0} d(y,0)^{-\eta}|u^{n}(y)| \leq 1 \label{e43}.
\end{equation}
Furthermore, due to the inequality $[(R^{n})^{-\eta}S^{R^{n}}_{x^{n}}U^{n} \mid_{x_{0}=0} ]_{G^{\eta}(\R^{d})} \leq \frac{1}{n}$, the third estimate can be improved along the initial time boundary $y_{0}=-(R^{n})^{-2}x^{n}_{0}$ to
\begin{equation}
    \sup_{y_{1} \in \R^{d}, y \neq 0} |y_{1}|_{\R^{d}}^{-\eta}|u^{n}(-(R^{n})^{-2}x^{n}_{0},y_{1})| \lesssim \frac{1}{n} \label{e44}.
\end{equation}
Furthermore, for $z_{n}:= \big  ( S_{x^{n}}^{R^{n}} \big )^{-1}y^{n}$
it holds $|u^{n}(z^{n})| \geq \frac{1}{2}$.  

\medskip

To obtain compactness, we can choose an extension $\hat{u}^{n}$ of $u^{n}$ to $\R^{d+1}$ while preserving the first uniform bound in \eqref{e43}, i.e., $[\hat u^n]_{C^\alpha(B_R)}\lesssim R^{\eta-\alpha}$.
In fact, by Theorem 2 of \cite{mcshane1934extension}, we may extend $u^n$ for each $R>0$ uniquely to $\overline{B_R\cap D^n}$ such that $[u^{n}]_{C^{\alpha}(\overline{B_{R} \cap D^{n}})}\lesssim [u^{n}]_{C^{\alpha}(B_{R} \cap D^{n})}\lesssim R^{\eta-\alpha}$.
Since $R>0$ is arbitrary, this yields a (unique) continuous continuation of $u^n$ to $\overline{D_n}$.
Then $\hat{u}^n$ can be defined as this continuation on $\overline{D_n}$ and by extending the values on the boundary of the the time-space strip $\overline{D_n}$ forwards and backwards in time, respectively.
By Arzel\`a-Ascoli's Theorem,  we find that $\hat{u}^{n}$ converges (along a subsequence) locally uniformly to a continuous limit $\hat{u}$ on $\R^{d+1}$.

\medskip

We now argue that $\hat{u}$ satisfies $\mathcal{L}\hat{u}=0$ on a suitable domain $D \subseteq \R^{d+1}$.  To this end, first observe that from \eqref{e43}, for all $\varphi \in \mathcal{B}$ and bounded sequences $(y^{n},\lambda^{n}) \in D^{n} \times \R_{+}$ such that $B_{\lambda^{n}}(y^{n}) \subseteq D^{n}$ it holds
\begin{equation}
    \lim_{n \to \infty} \langle \mathcal{L} u^{n},\varphi^{\lambda^{n}}_{y^{n}} \rangle =0. \label{e45}
\end{equation}
This is an immediate consequence of taking $R$ sufficiently large and using the definition of the $C^{\alpha-2}(B_{R} \cap D^{n})$ semi-norm.  Passing to a further subsequence if necessary, we may assume that $-(R^{n})^{-2}x^{n}_{0}$ converges to $T_{1} \in [-\infty, 0]$ and $(R^{n})^{-2}(T^{n}-x^{n}_{0})$ converges to $T_{2} \in [0,\infty]$.  Note if $T_{1}>-\infty$, then by \eqref{e44} it holds that $\hat{u}(T_{1},\cdot)=0$ and hence if $T_{1}=T_{2}=0$, then $\hat{u}(0,\cdot)$ is identically zero.  However, in this case $z_{0}^{n} \to 0$ so that $|\hat{u}(0,z_{1})| \geq \frac{1}{2}$ for some $z_{1} \neq 0$, a contradiction. 

\medskip

Therefore, we may assume $T_{1}<T_{2}$ and we define $D=(T_{1},T_{2}) \times \R^{d}$ and claim that for all $\varphi \in \mathcal{B}$ and $(y,\lambda) \in D \times \R_{+}$ such that $B_{\lambda}(y) \subseteq D$ it holds
\begin{equation}
     \langle \mathcal{L} \hat{u} ,\varphi^{\lambda }_{y} \rangle =0. \label{e46}
\end{equation}
Indeed, this follows easily from \eqref{e45} by choosing a sequence $(y^{n},\lambda^{n}) \in D^{n} \times \R_{+}$ converging to $(y,\lambda)$ and satisfying $B_{\lambda^{n}}(y^{n}) \subseteq D^{n}$.  Using the compact support of $\varphi$, the local uniform convergence of $\hat{u}^{n}$ (which agrees with $u^{n}$ on the support of $\varphi^{\lambda^{n}}_{y^{n}}$) to $\hat{u}$, and the pointwise convergence of $\mathcal{L}^{*}\varphi^{\lambda^{n}}_{y^{n}}$ to $\mathcal{L}^{*}\varphi^{\lambda}_{y}$, we find that sending $n \to \infty$ to \eqref{e45} yields \eqref{e46}.  Using \eqref{e46} to take $\lambda \to 0$ as in the proof of Theorem \ref{thm:continuum}, we find that for all $\varphi \in C^{\infty}_{c}(D)$ it holds $\langle \mathcal{L} \hat{u} ,\varphi \rangle=0$, and therefore by ellipticity of $\mathcal{L}$ we find that $\hat{u}$ is a smooth solution to $\mathcal{L}\hat{u}=0$ in $D$.

\medskip

We now contradict Liouville's theorem.
On one hand, by construction (and possibly passing to a further subsequence), there exists $z \in \overline{D}$ with $d(z,0)=1$ such that $|\hat{u}(z)| \geq \frac{1}{2}$, and in particular $\hat{u}$ cannot be identically zero in $D$ (if $z \in \partial D$ this would contradict continuity of $\hat{u}$ on $\R^{d+1}$). 
On the other hand, we have that $\hat{u}$ is a classical solution to $\mathcal{L}\hat{u}=0$ in $D$, which by the third point in \eqref{e43} and the local uniform convergence satisfies the growth constraint $|\hat{u}(y)| \leq d(y,0)^{\eta}$ for $\eta \in (1,2)$.
We claim that for $T_{1}<T_{2}$ this violates Liouville's theorem.
Indeed, if $T_{1}>-\infty$, then again by \eqref{e44} it holds that $\hat{u}(T_{1}, \cdot )=0$ and therefore $\hat{u}(x_{0},\cdot)$ is identically zero for all $x_{0} \in (T_1,T_{2})$ by standard uniqueness results for the heat equation for solutions with polynomial growth, see Theorem 7 in Chapter 2.3 of \cite{Eva10}.
If instead $T_{1}=-\infty$, then $\hat{u}$ is an ancient solution to the heat equation, so Liouville's theorem for ancient solutions (see \cite{Nic37}, or \cite[Theorem 1.2(b)]{LiZ19} for a more modern exposition) implies $\hat{u}$ is a restriction of a polynomial to $D$, which gives a contradiction. 
\end{proof}

\subsection{The Discrete Laplacian on $\epsilon \Z^{d}$}
In this section, we consider the necessary modifications of Simon's method to analyze the discrete setting $\epsilon \Z^{d} \subseteq \R^{d}$.  For a detailed study of discrete approximations in singular SPDE, we refer the reader to the works \cite{hairer2018discretisations}, \cite{erhard2019discretisation}, and \cite{martin2019paracontrolled}.   We choose $\mathfrak{s}=(1,\cdots,1)$ and take the principal operator to be the discrete Laplacian,
defined for $u: \epsilon \Z^{d} \mapsto \R$ by
\begin{equation}
    y \in \epsilon \Z^{d} \mapsto \Delta_{\epsilon}u(y):=\epsilon^{-2} \sum_{z \in \epsilon \Z^{d}, d(y,z)=\epsilon}(u(z)-u(y)) \label{e21}.
\end{equation}
For $0<\alpha<\eta$, the $G^{\eta}(\epsilon \Z^{d})$ and $G^{\eta,\alpha}(\epsilon \Z^{d})$ semi-norms have already been defined in \eqref{e30} and \eqref{e41}.
However, since $\epsilon \Z^{d}$ is not an open subset of $\R^{d}$, we have not yet defined the $G^{\gamma}(\epsilon \Z^{d})$ semi-norm for $\gamma<0$.
We define this in analogy with the discrete negative H\"{o}lder norm used in \cite{hairer2018discretisations}.  Namely, we define the bilinear form $\langle \cdot, \cdot \rangle_{\epsilon}$ via
\begin{equation}
    \langle f,g\rangle_{\epsilon}:=\epsilon^{d}\sum_{x \in \epsilon \Z^{d}}f(x) g(x), \nonumber
\end{equation}
and for a germ $V$ over $\epsilon \Z^{d}$ we set
\begin{equation}
[V]_{G^{\gamma}(\epsilon \Z^{d})}:=\sup_{\varphi \in \mathcal{B}}\sup_{x \in \epsilon \Z^{d},\lambda \geq \epsilon} \lambda^{-\gamma}\big |\langle V_{x},\varphi_{x}^{\lambda} \rangle_{\epsilon} \big | \label{e24}.
\end{equation}
In comparison to the continuous setting, we require a further centering assumption on the germs.  We say that a germ $U$ on $\epsilon \Z^{d}$ is \textit{centered} (to order one) provided that $U_{x}(x)=0$ and $D_{j}U_{x}(x)=0$ for all $x\in \epsilon \Z^{d}$ and $j \in [d]$, where $D_{j}$ is the forward difference operator on $\epsilon \Z^{d}$, i.e.~$D_{j}U_{x}(x)=\epsilon^{-1} \big (U_{x}(x+\epsilon \mathbf{e}_{j})-U_{x}(x) \big )$.\footnote{In the continuum setting, the centering condition is automatically implied whenever the $G^{\eta}(\R^{d})$ norm is finite for $\eta \in (1,2)$, where $D_{j}$ is replaced by the continuum directional derivative $\partial_{j}$.} The Schauder estimate for germs is now stated as follows.
\begin{theorem}\label{thm:discrete}
Let $0<\alpha<1<\eta<2$.  There exists a constant $C>0$ such that for all $\epsilon>0$ and centered germs $U$ on $\epsilon \Z^{d}$ with $\|U\|_{G^{\eta}(\epsilon \Z^{d})}<\infty$ it holds 
\begin{equation}
    \|U\|_{G^{\eta}(\epsilon \Z^{d})} \leq C \big (  [\Delta_{\epsilon}U]_{G^{\eta-2}(\epsilon \Z^{d})}+[U]_{G^{\eta,\alpha}(\epsilon \Z^{d})} \big ) \label{e4}.
\end{equation}
\end{theorem}
As with the proof of Theorem \ref{thm:wholespace}, the scaling of the above semi-norms plays an important role in establishing Theorem \ref{thm:discrete} and we need the following discrete analogue of \eqref{e53}, which can be proved based on a change of variables and the identity
\begin{equation}
 (\Delta_{\epsilon R^{-1}}S_{w}^{R}U)_{x}=R^{2}(\Delta_{\epsilon}U)_{S^{R}_{w}x} \circ S^{R}_{w} \nonumber.
\end{equation}

\begin{lemma}\label{lemma:rescaling4}
Let $R>0$, $w \in \epsilon \Z^{d}$, and $\gamma<0$. Then for all germs $U$ over $\epsilon \Z^{d}$ 
\begin{equation}
 [\Delta_{\epsilon R^{-1}}S^{R}_{w}U]_{G^{\gamma}(\epsilon R^{-1}\Z^{d})}=R^{2+\gamma}[\Delta_{\epsilon}U]_{G^{\gamma}(\epsilon  \Z^{d})} \label{e7}.
\end{equation}
\end{lemma}
We now comment briefly on the proof of Theorem \ref{thm:discrete}.  By Lemma \ref{lemma:Rescaling} and Lemma \ref{lemma:rescaling4}, we can reduce to the case $\epsilon=1$.  The basic strategy is then the same as the proof of Theorem \ref{thm:continuum}.  The difference now is that after zooming in near a point of concentration, we will have a sequence of functions on lattices of variable width.  If the width stays bounded away from zero, we will extract (along a subsequence) a limiting discrete harmonic function and contradict the discrete Liouville theorem on $\Z^{d}$.
On the other hand, if the lattice spacing goes to zero, we will use the extension operators from \cite{martin2019paracontrolled} to extract a limit and contradict Liouville's theorem in the continuum, as in Theorem \ref{thm:continuum}.

\begin{proof}[Proof of Theorem \ref{thm:discrete}]
By re-scaling and using Lemma \ref{lemma:rescaling4} we may assume $\epsilon=1$.  Indeed, assume the estimate \eqref{e4} is established for arbitrary centered germs on the unit lattice.  Given a centered germ $U$ over $\epsilon \Z^{d}$, applying \eqref{e4} to the centered germ $S^{\epsilon}_{0}U$
over $\Z^{d}$ yields the desired bound in light of Lemmas \ref{lemma:Rescaling} and \ref{lemma:rescaling4}.  

\medskip

If the statement is false for $\epsilon=1$, there exists a sequence of germs $\tilde{U}^{n}$ on $\Z^{d}$ and  $x^{n},y^{n} \in \Z^{d}$ with $x^n\ne y^n$ such that for $R^{n}:=d(x^{n},y^{n})$, it holds that
\begin{align}
    \|\tilde{U}^{n}\|_{G^{\eta}(\Z^{d})}=1, \qquad (R^{n})^{-\eta}|\tilde{U}_{x_{n}}(y_{n})| \geq \frac{1}{2}, \qquad [\Delta_{1}\tilde{U}^{n}]_{G^{\eta-2}(\Z^{d})}+[\tilde{U}^{n}]_{G^{\eta,\alpha}(\Z^{d})} \leq \frac{1}{n} \label{e5}.
\end{align}
We define $\epsilon^{n}:=(R^{n})^{-1}$ and a new sequence of germs $U^{n}:=(R^{n})^{-\eta}S^{R^{n}}_{x^{n}}\tilde{U}$ on $\epsilon^{n}\Z^{d}$.
By Lemma \ref{lemma:Rescaling}\ref{lemma:rescaling4}, the identity \eqref{e5} now holds with $\Z^{d}$ replaced by $\epsilon^{n} \Z^{d}$, that is
\begin{align}
    \|U^{n}\|_{G^{\eta}(\epsilon^{n}\Z^{d})}=1, \qquad |U^{n}_{0}(z^{n})| \geq \frac{1}{2}, \qquad [\Delta_{\epsilon^{n}} U^{n}]_{G^{\eta-2}(\epsilon^{n}\Z^{d})}+[U^{n}]_{G^{\eta,\alpha}(\epsilon^{n}\Z^{d})} \leq \frac{1}{n} \label{e23},
\end{align}
where $z^{n}:=(R^{n})^{-1}(y^{n}-x^{n})$.  We now define a sequence of functions on $\epsilon^{n}\Z^{d}$ by $u^{n}:=U_{0}^{n}$ .  Note that by \eqref{e23}, for all $z \in \epsilon^{n}\Z^{d}$ it holds
\begin{equation}
    |u^{n}(z)| \leq d(0,z)^{\eta} \label{e19}.
\end{equation}
To obtain a uniform H\"{o}lder bound, we apply  Lemma \ref{lemma:Holderbound2}, noting that $\epsilon^{n}\Z^{d}$ has the property  that for $x,y \in \epsilon^{n} \Z^{d}$ it holds $d(x,y) \in \epsilon^{n} \Z$, cf. \eqref{e35} and hence $y+d(x,y)e_{i} \in \epsilon^{n} \Z^{d}$ for all $i \in [d]$.  Thus, we find that for all $R>0$
\begin{equation}
    [u^{n}]_{C^{\alpha}(\epsilon^{n}\Z^{d} \cap B_{R})} \lesssim R^{\eta-\alpha}.\label{e47}
\end{equation}

Before extracting a suitable limit, we now argue that for $\varphi \in \mathcal{B}$, $y \in \epsilon^{n}\Z^{d}$ and $\lambda \geq \epsilon^{n}$, 
\begin{equation}
    \lim_{n \to \infty} \langle \Delta_{\epsilon^{n}}u^{n}, \varphi^{\lambda}_{y} \rangle_{\epsilon^{n}}=0 \label{e25}.
\end{equation}
In fact, this is a straightforward adaptation of the corresponding argument in the proof of Lemma \ref{lemma:Holderbound3}, which we give for the convenience of the reader.
\begin{equation}
    \langle \Delta_{\epsilon^{n}} u^{n}, \varphi_{y}^{\lambda} \rangle_{\epsilon^{n}}= \langle \Delta_{\epsilon^{n}} U^{n}_{0}, \varphi_{y}^{\lambda} \rangle_{\epsilon^{n}}=\langle \Delta_{\epsilon^{n}} U^{n}_{y}, \varphi_{y}^{\lambda} \rangle_{\epsilon^{n}}+\langle \Delta_{\epsilon^{n}} (U^{n}_{0}-U^{n}_{y}), \varphi_{y}^{\lambda} \rangle_{\epsilon^{n}}.
\end{equation}
By definition \eqref{e24} together with \eqref{e23}, it holds $|\langle \Delta_{\epsilon^{n}} U^{n}_{y}, \varphi_{y}^{\lambda} \rangle_{\epsilon^{n}}| \leq \frac{1}{n}\lambda^{\eta-2} \to 0$. For the second term, 
\begin{align}
  \big |\langle \Delta_{\epsilon^{n}} (U^{n}_{0}-U^{n}_{y}), \varphi_{y}^{\lambda} \rangle_{\epsilon^{n}} \big |=&\bigg |(\epsilon^{n})^{d}\sum_{z \in \epsilon^{n}\Z^{d}} (U_{0}^{n}(z)-U_{0}^{n}(y)-U^{n}_{y}(z)+U_{y}(y) ) \Delta_{\epsilon_{n}} \varphi_{y}^{\lambda}(z) \bigg |  \nonumber \\
  &\leq [U^{n}]_{G^{\eta,\alpha}(\epsilon^{n}\Z^{d})} (\epsilon^{n})^{d} \sum_{z \in \epsilon^{n}\Z^{d}} d(z,y)^{\alpha} \big (d(0,y)+d(z,y) \big )^{\eta-\alpha} |\Delta_{\epsilon^{n}} \varphi_{y}^{\lambda}(z)| \to 0 \nonumber.
\end{align}

\medskip

We now proceed to extract a limit and note that since $R^{n} \in \Z$,  $\epsilon^{n}$ is a bounded sequence admitting a subsequential limit $\epsilon^{\infty}$ (not relabelled).  There are now two cases to consider:  $\epsilon^{\infty}=0$ and  $\epsilon^{\infty}>0$. 

\medskip

We start with the case $\epsilon^{\infty}=0$ and define an extension $\hat{u}^{n}$ of $u^{n}$ from $\epsilon^{n}\Z^{d}$ to $\R^{d}$ while preserving the uniform bound \eqref{e47}.  This can be done using the extension operators defined in \cite{martin2019paracontrolled} (see also the Appendix of \cite{gubinelli2021pde}).  Since $\hat{u}^{n}$ extends $u^{n}$, we find that for all $y \in \epsilon^{n}\Z^{d}$ it holds
\begin{equation}
|\hat{u}^{n}(y)| \leq d(0,y)^{\eta} \label{e26}.
\end{equation}
By \eqref{e47} and Arzel\`a-Ascoli's theorem, it follows that $(\hat{u}^{n})_{n}$ converge locally uniformly (along a subsequence) to a continuous function $\hat{u}$.  We now argue that $\hat{u}$ is harmonic, for which (as in the proof of Theorem \ref{thm:continuum}) it suffices to show that for any $y \in \R^{d}$ and $\lambda>0$ it holds $\langle \Delta u, \varphi_{y}^{\lambda} \rangle=0$.  To this end, we take a sequence $y^{n} \in \epsilon^{n}\Z^{d}$ converging to $y$ and $\lambda^{n} \geq \epsilon^{n}$ converging to $\lambda$, then by \eqref{e25}, using that $\hat{u}^{n}$ agrees with $u^{n}$ on $\epsilon^{n}\Z^{d}$, we have
\begin{equation}
 \lim_{n \to \infty} \langle \Delta_{\epsilon^{n}}\hat{u}^{n},\varphi_{y^{n}}^{\lambda^{n}} \rangle_{\epsilon^{n}}=0 \nonumber.   
\end{equation}
However, by Lebesgue's dominated convergence theorem, it holds that $\langle \Delta_{\epsilon^{n}}\hat{u}^{n},\varphi_{y^{n}}^{\lambda^{n}} \rangle_{\epsilon^{n}}$ converges to $\langle \Delta u, \varphi_{y}^{\lambda} \rangle$, yielding that $\hat{u}$ is harmonic (as in the proof of Theorem \ref{thm:continuum}).  To obtain a contradiction, we only need to argue the growth bound: for all $y \in \R^{d}$ it holds
\begin{equation}
    |\hat{u}(y)| \leq  d(0,y)^{\eta} \nonumber.
\end{equation}
Indeed, this follows from choosing $y^{n} \in \epsilon^{n}\Z^{d}$ converging to $y$, then passing to the limit on both sides of \eqref{e26}, using the uniform convergence of $\hat{u}^{n}$ to the continuous limit $\hat{u}$.  Hence, $\hat{u}$ must be zero by Liouville's theorem, but $|\hat{u}^{n}(z^{n})|=|u^{n}(z^{n})| \geq \frac{1}{2}$, yielding a contradiction.

\medskip

We now turn to the case $\epsilon^{\infty}>0$, which in particular means that $\epsilon^n=\epsilon^\infty$ for sufficiently large $n$ due to $R^n=(\epsilon^n)^{-1} \in \Z$.
We start by extending to $\R^{d}$ as above and extract a limit $\hat{u}$ as before.  In this case, we find that for all $y \in \epsilon^{\infty}\Z^{d}$ and $\lambda \geq \epsilon^{\infty}$, it holds that $\langle \Delta_{\epsilon^{\infty}} \hat{u},\varphi^{\lambda}_{y}\rangle_{\epsilon^{\infty}}=0$.  Since $\varphi \in \mathcal{B}$ and $\mathfrak{s}=(1,\cdots, 1)$, it is supported in the unit $L^{1}$ ball on $\Z^{d}$ with $\varphi(0)=1$, so setting $\lambda=\epsilon^{\infty}$, we find that $\langle \Delta_{\epsilon^{\infty}} \hat{u},\varphi^{\lambda}_{y}\rangle_{\epsilon_{\infty}}=\Delta_{\epsilon^{\infty}} \hat{u}(y) \varphi^{\lambda}(0)=\Delta_{\epsilon^{\infty}} \hat{u}(y) (\epsilon^{\infty})^{-d}$, and hence the restriction of $\hat{u}$ to $\epsilon^{\infty}\Z^{d}$ is discretely harmonic.  Moreover, it holds $|u(y)| \leq d(0,y)^{\eta}$ and by the centering condition on the approximating germ, $u$ vanishes on the vertices of the standard simplex in $\epsilon^\infty\Z^d$, so by the discrete Liouville theorem on $\epsilon^{\infty}\Z^{d}$ in Lemma \ref{js1000}, we obtain $u=0$.
But $|u^n(z^n)|\ge \frac12$ for all $n$, where $z^n\in \epsilon^n\Z^d=\epsilon^\infty\Z^d$ fulfill $d(0,z^n)=1$.
Passing to a subsequential limit $z\in \epsilon^\infty\Z^d$, we arrive at the contradiction $|u(z)|\ge \frac12$.
\end{proof}

\section{Some Further Extensions}\label{sec:further}
In this section, we explore some generalizations of the results in Section \ref{sec:illustration} to further showcase the flexibility of Simon's method.

\subsection{Elliptic Operators with Constant Coefficients}\label{sec:generalOperators}

We now generalize Theorem \ref{thm:continuum} from $\Delta$ to an arbitrary scaling homogeneous elliptic operator $\mathcal{L}$ with constant coefficients.

\begin{theorem}\label{thm:wholespace}
Let $m\in\N$ and let $\mathcal{L}=\sum_{|\gamma|=m} a_\gamma \partial^\gamma$ be an elliptic operator of anisotropic order $m$.
Suppose $0<\alpha<\eta<m$ with $\alpha,\eta \notin \N$.
There exists a constant $C$ such that for all germs $U$ with $\|U\|_{G^\eta(\R^{d})}<\infty$ it holds
\begin{equation}
\|U\|_{G^\eta(\R^d)} \leq C \big (  [\mathcal{L}U]_{G^{\eta-m}(\R^d)}+[U]_{G^{\eta,\alpha}(\R^d)} \big ). \label{G1}
\end{equation}
\end{theorem}

\begin{example}
Let us give some examples of operators that fit into this framework.
As observed in \cite{simon1997schauder}, an operator $\mathcal{L}=\sum_{|\gamma|=m} a_\gamma \partial^\gamma$ is elliptic if and only if $\sum_{|\gamma|=m}a_\gamma (i \xi)^\gamma \ne 0$ for all $\xi\in\R^d\setminus\{0\}$.
Therefore, the following examples are readily seen to be elliptic.
    \begin{enumerate}
        \item The Laplace operator: $\mathcal{L}=\Delta$, $m=2$ and $\mathfrak{s}:=(1,\ldots,1)$.
        \item The Cauchy-Riemann operator: $\mathcal{L}=\frac12(\partial_{x_1}+i\partial_{x_2})$, $m=1$ and $\mathfrak{s}:=(1,\ldots,1)$,
        or more generally $\mathcal{L}=\alpha\partial_{x_1}+\beta\partial_{x_2}$ with $\alpha,\beta\in\C$ being linearly independent over $\R$.
        \item Heat operator: $\mathcal{L}=\partial_{x_1}-\sum_{j=2}^{d} \partial_{x_j}^2$, $m=2$ and $\mathfrak{s}:=(2,1,\ldots,1)$.
    \end{enumerate}
\end{example}
\begin{remark}\label{remark:classical}
To make a comparison with classical Schauder theory, let $u: \R^{d} \mapsto \R$ be a smooth function and for each base-point $x$, let $z \mapsto Q_{x}(z)$ be in $\mathcal{P}_{\lfloor \eta\rfloor}$. In this case, the germ $U$ defined by $U_{x}:=u-Q_{x}$ satisfies $[U]_{G^{\eta,\alpha}(D) }=0$.  Indeed, for each $x,y \in D$, defining $P_{x,y} \in \mathcal{P}_{\lfloor \eta\rfloor}$ by $P_{x,y}:=Q_{y}-Q_{x}$, we automatically have $ U_{x}-U_{y}-P_{x,y}$ is identically zero.  In particular, this holds if we take $Q_{x}$ to be Taylor jet of $u$ to order $\lfloor \eta\rfloor$ near $x$.
\end{remark}

We now turn to the proof.
\begin{proof}[Proof of Theorem \ref{thm:wholespace}]
The proof follows an identical line of argument as in the proof of Theorem \ref{thm:continuum}, with the following modifications.
First replace $\Delta$ by $\mathcal{L}$ and write $m$ instead of $2$ for the anisotropic order of this operator.
For the compactness argument, we use a generalization of Lemma \ref{lemma:Holderbound2} provided by Lemma \ref{lemma:Holderbound_gen_hyb} (which removes the restriction $0<\alpha<1<\eta<2$).
To verify the limit is a classical solution to $\mathcal{L}u=0$ on $\R^{d}$, we use the fact that if $\langle \mathcal{L} u, \varphi_{y}^{\lambda} \rangle=0$ for all $y \in \R^{d}$, $\lambda>0$ then $u$ is $\mathcal{L}$-harmonic as a distribution and thus a smooth function by ellipticity of $\mathcal{L}$, cf. Theorem 4.1.4 in \cite{Hor83}. 
Finally, the contradiction comes from the generalized Liouville theorem provided below in Lemma \ref{js1000}, since $u$ is polynomially bounded by Lemma \ref{lemma:Holderbound_gen_hyb} and thus a tempered distribution.  The assumption that $\eta \notin \N$ is used to conclude that if $u$ is an $\mathcal{L}$-harmonic polynomial satisfying $|u(z)| \lesssim d(z,0)^{\eta}$, then $u$ is identically zero.  
\end{proof}

We now generalize Lemma \ref{lemma:Holderbound2} from Section \ref{sec:illustration}, removing the restrictions $\alpha<1$ and $\eta<2$.
The argument depends on precise estimates for the higher order polynomial contributions, which we provide below in Lemma \ref{lemma:polynomialbound_hyb}.

\begin{lemma}\label{lemma:Holderbound_gen_hyb}
Let $0<\alpha<\eta$ with $\alpha\notin\N$.
 Let $D\subseteq \R^d$ be such that for all $x,y\in D$ with $x\ne y$ it holds $y+\sum_{j=1}^d (\rho_j d(x,y))^{\mathfrak{s}_j}\mathbf{e}_j\in D$ for all $\rho=(\rho_1,\ldots,\rho_d)\in \N^d$.
Assume that $U$ is a germ with $\|U\|_{G^{\eta}(D)}+[U]_{G^{\eta,\alpha}(D)}<\infty$.
Then for all $x\in D$ and $y\in D\cap B_R(x)$ it holds
 \begin{align}
     \sup_{x\in D}[U_{x}]_{C^{\alpha}(D\cap B_{R}(x))} \lesssim_{\mathfrak{s},d,\eta,\alpha} \big ( \|U\|_{G^{\eta}(D)}+[U]_{G^{\eta,\alpha}(D)} \big ) R^{\eta-\alpha} \label{e14a_hyb}.
 \end{align}
\end{lemma}
\begin{proof}
Fix a base point $x \in D$ and a radius $R>0$.
Since $[U]_{G^{\eta,\alpha}(D)}<\infty$, for each $y \in D\cap B_{R}(x)$ there exists a polynomial $P_{xy} \in \mathcal{P}_{\lfloor\eta\rfloor}$ such that for all $z \in D\cap B_{R}(x)$ we have
\begin{align}\label{e12a_hyb}
\begin{split}
\big |(U_{x}-U_{y}-P_{xy})(z) \big | 
&\leq 2[U]_{G^{\eta,\alpha}(D)}  d(y,z)^{\alpha}\big (d(y,z)+d(y,x) \big )^{\eta-\alpha} \\
& \lesssim [U]_{G^{\eta,\alpha}(D)}  d(y,z)^{\alpha}R^{\eta-\alpha}.
\end{split}
\end{align}
Observe that for $x=y$ we may choose $P_{xx}=0$.
For $x\ne y$ we write $P_{xy}(z)=:\sum_{|\beta|\le \eta} \nu_{x\beta}(y)(z-y)^\beta$.
Since $\|U\|_{G^{\eta}(D)}<\infty$, we have  $|U_{y}(z)| \leq \|U\|_{G^{\eta}(D)}d(y,z)^{\eta} \lesssim \|U\|_{G^{\eta}(D)} d(y,z)^{\alpha}R^{\eta-\alpha}$.
Hence, by the triangle inequality, \eqref{e12a_hyb} turns into
\begin{equation}
\big |U_{x}(z)- P_{xy}^{<\alpha}(z) \big | 
\lesssim \big(  [U]_{G^{\eta,\alpha}(D)}+\|U\|_{G^{\eta}(D)}   \big )d(y,z)^{\alpha}R^{\eta-\alpha}+|P_{xy}^{\ge\alpha}(z-y)|,\label{e13a_hyb}
\end{equation}
where $P_{xy}^{<\alpha}(z):=\sum_{|\beta|< \alpha} \nu_{x\beta}(y)(z-y)^\beta$ and $P_{xy}^{\ge\alpha}:=P_{xy}-P_{xy}^{<\alpha}$.
Observe that $P_{xy}^{<\alpha}\in \mathcal{P}_{\lfloor\alpha\rfloor}$, and that $[f]_{C^\alpha(D\cap B_R(x))}$ is the infimum over all $M>0$ such that for all $y\in D\cap B_R(x)$ there is a $P\in\mathcal{P}_{\lfloor\alpha\rfloor}$ such that for all $z\in D\cap B_R(x)$ it holds $|f(z)-P(z)|\le M d(z,y)^\alpha$, so that it remains only to estimate $P_{xy}^{\ge\alpha}(z)$ in order to establish \eqref{e14a_hyb}.
For this we use
\begin{align*}
 |P_{xy}^{\ge\alpha}(z)| \leq \sum_{\alpha \le |\beta|\le \eta}|\nu_{x\beta}(y)|d(z,y)^{|\beta|} \lesssim \sum_{\alpha\le |\beta|\le \eta} |\nu_{x\beta}(y)| d(z,y)^{\alpha}R^{|\beta|-\alpha}.   
\end{align*}
Lemma \ref{lemma:polynomialbound_hyb} below yields $\sum_{\alpha\le |\beta|\le \eta} |\nu_{x\beta}(y)|\lesssim ([U]_{G^{\eta,\alpha}(D)}+[U]_{G^{\eta}(D)}) R^{\eta-|\beta|}$.
Consequently we have
\begin{align*}
    |P_{xy}^{\ge\alpha}(z)|\lesssim \big(  [U]_{G^{\eta,\alpha}(D)}+[U]_{G^{\eta}(D)}   \big )d(y,z)^{\alpha}R^{\eta-\alpha}.
\end{align*}
Combining this with \eqref{e13a_hyb} gives \eqref{e14a_hyb}.
\end{proof}

\begin{lemma}\label{lemma:polynomialbound_hyb}
 Let $0<\alpha<\eta$.
 Let $D\subseteq \R^d$ and $x,y\in D$, $x\ne y$ be such that $y+\sum_{j=1}^d (\rho_j(d(x,y))^{\mathfrak{s}_j}\mathbf{e}_j\in D$ for all $\rho=(\rho_1,\ldots,\rho_d)\in \N^d$.
 Assume that $U$ is a germ with $\|U\|_{G^\eta(D)}+[U]_{G^{\eta,\alpha}(D)}<\infty$, and suppose that there is $(\nu_\beta)_{|\beta|\le\eta}\subset \R$.
 For $z\in D$ write $P(z):=\sum_{|\beta|\le \eta} \nu_{\beta} (z-y)^\beta$.
 If there is $C\ge 1$ such that for all $z\in D$ it holds
 \begin{align*}
  |(U_x-U_y-P)(z)|\le C[U]_{G^{\eta,\alpha}(D)}d(y,z)^\alpha\big(d(x,y)+d(y,z)\big)^{\eta-\alpha},
 \end{align*}
 then for all $|\beta|\le \eta$ it holds
 \begin{align*}
  |\nu_{\beta}|\lesssim_{\mathfrak{s},d,\eta,\alpha} (\|U\|_{G^\eta(D)}+C[U]_{G^{\eta,\alpha}(D)}) d(x,y)^{\eta-|\beta|}.
 \end{align*}
\end{lemma}
\begin{proof}
 Within the proof $\lesssim$ and $\simeq$ will always stand for an inequality (resp.\@ equality) up to a constant depending only on the scaling $\mathfrak{s}$, the dimension $d$, the regularities $\eta$ and $\alpha$, and the parameter $c$.
 Write $\mathsf{A}$ for the set of $\beta$'s with $|\beta|\le \eta$ and $|\mathsf{A}|$ for its cardinality.
 Let $x,y\in \R^d$ be given.
 For $\eps\in \N$ and $\rho=(\rho_1,\ldots,\rho_d)\in \N^d$ the assumption on $D$ implies that $z:=y+\sum_{j=1}^d (\eps\rho_j d(x,y))^{\mathfrak{s}_j} \mathbf{e}_j\in D$.
 In particular, $|(z-y)^\beta|= \eps^{|\beta|}\rho^{\mathfrak{s}\beta}d(x,y)^{|\beta|}$ for all $\beta\in\mathsf{A}$,\footnote{We use the notation $\rho^{\mathfrak{s}\beta}:=\prod_{j=1}^d \rho_j^{\mathfrak{s}_j\beta_j}$.} and moreover $d(z,x)\le d(x,y)+ d(y,z)\lesssim c_{\eps,\rho}d(x,y)$.
We hence have
 \begin{align*}
  \eps^{|\beta|}\rho^{\mathfrak{s}\beta}d(x,y)^{|\beta|}|\nu_{\beta}| & \simeq |\nu_{\beta}(z-y)^\beta| \\
  &\lesssim |(U_x-U_y-P)(z)|+|U_x(z)|+|U_y(z)|+\sum_{\gamma\in\mathsf{A}\setminus\{\beta\}} \eps^{|\gamma|}\rho^{\mathfrak{s}\gamma}d(x,y)^{|\gamma|}|\nu_{\gamma}|  \\
  &\lesssim \big(c_{\eps,\rho}(C[U]_{G^{\eta,\alpha}(D)} + \|U\|_{G^\eta(D)}) + \sum_{\gamma\in\mathsf{A}\setminus\{\beta\}} \eps^{|\gamma|}\rho^{\mathfrak{s}\gamma}d(x,y)^{|\gamma|-\eta} |\nu_{\gamma}| \big) d(x,y)^{\eta}.
 \end{align*}
 Thus it follows that for all $\eps\in\N$ and $\rho\in \N^d$ we have a constant $c_{\eps,\rho}>0$ such that
 \begin{align}
  \frac{|\nu_{\beta}|}{d(x,y)^{\eta-|\beta|}}\lesssim c_{\eps,\rho} (C[U]_{G^{\eta,\alpha}(\R^d)} \|U\|_{G^\eta(\R^d)}) + \sum_{\gamma\in\mathsf{A}\setminus\{\beta\}} \eps^{|\gamma|-|\beta|}\rho^{\mathfrak{s}(\gamma-\beta)} \frac{|\nu_{\gamma}|}{d(x,y)^{\eta-|\gamma|}}. \label{e50_hyb}
 \end{align}
 Here $\rho^{\mathfrak{s}(\gamma-\beta)}=\prod_{j=1}^d \rho_j^{\mathfrak{s}_j(\gamma_j-\beta_j)}$ is well defined regardless of the sign of $\gamma_j-\beta_j$, i.e.\@ some components of $\rho$ might appear with a positive power, while some appear with a negative power.
 If $|\mathsf{A}|=1$, the sum on the right is empty, and the result follows upon taking $\eps=1$ and $\rho=(1,\ldots,1)$;
 we thus may assume $|\mathsf{A}|>1$.

 \medskip
 
 We now make the following claim: for any $\delta>0$, there exists a collection $(\kappa_{\beta}, \eps_{|\beta|}, \rho_{\beta})_{\beta \in \mathsf{A}}$ with 
 $\kappa_\beta\in [1,\infty)$, $\eps_{|\beta|}\in \N$, and $\rho_\beta\in \N^d$ such that for all $\gamma\in\mathsf{A}$
 \begin{align}\label{js150_hyb}
  \sum_{\beta\in\mathsf{A}\setminus\{\gamma\}} \kappa_\beta \eps_{|\beta|}^{|\gamma|-|\beta|}\rho_\beta^{\mathfrak{s}(\gamma-\beta)} \le \delta \kappa_\gamma.
 \end{align}
We temporarily postpone the proof of the above claim, and use \eqref{js150_hyb} to complete the proof.
The idea is to sum over a weighted version of \eqref{e50_hyb} in order to buckle the estimate.
Namely, for each $\beta \in \mathsf{A}$ apply \eqref{e50_hyb} with $\eps,\rho$ played by $\eps_{|\beta|},\rho_{\beta}$ and multiply the result by $\kappa_{\beta}$. Summing over $\beta \in \mathsf{A}$ yields
 \begin{align*}
 \sum_{\beta\in\mathsf{A}} \kappa_\beta \frac{|\nu_{\beta}|}{d(x,y)^{\eta-|\beta|}} &\lesssim c_{\kappa,\eps,\rho}(C[U]_{G^{\eta,\alpha}(\R^d)} + \|U\|_{G^\eta(\R^d)}) + \sum_{\beta\in\mathsf{A}} \sum_{\gamma\in\mathsf{A}\setminus\{\beta\}} \kappa_\beta\eps_{|\beta|}^{|\gamma|-|\beta|}\rho_\beta^{\mathfrak{s}(\gamma-\beta)} \frac{|\nu_{\gamma}|}{d(x,y)^{\eta-|\gamma|}} \\
 &= c_{\kappa,\eps,\rho}(C[U]_{G^{\eta,\alpha}(\R^d)} + \|U\|_{G^\eta(\R^d)}) + \sum_{\gamma\in\mathsf{A}} \frac{|\nu_{\gamma}|}{d(x,y)^{\eta-|\gamma|}} \sum_{\beta\in\mathsf{A}\setminus\{\gamma\}} \kappa_\beta\eps_{|\beta|}^{|\gamma|-|\beta|}\rho_{\beta}^{\mathfrak{s}(\gamma-\beta)}  \\
 &\le c_{\kappa,\eps,\rho}(C[U]_{G^{\eta,\alpha}(\R^d)} + \|U\|_{G^\eta(\R^d)}) + \delta \sum_{\gamma\in\mathsf{A}} \kappa_\gamma \frac{|\nu_{\gamma}|}{d(x,y)^{\eta-|\gamma|}},
 \end{align*}
 where the constant $c_{\kappa,\eps,\rho}$ depends on all $\kappa_\beta$, $\eps_{|\beta|}$, and $\rho_\beta$ with $\beta$ running through $\mathsf{A}$.
 Relabelling $\gamma\mapsto \beta$ in the last summation and choosing $\delta\ll_{\mathfrak{s},d,\eta,\alpha} 1$, we obtain the assertion by absorbing.

 \medskip
 
 It remains to give the proof of the claim leading to \eqref{js150_hyb}.\footnote{On a first reading, we suggest the reader consider the case $d=1$, in which case the lexicographic ordering is unnecessary and the claim holds with $\rho_{\beta}=1$ for all $\beta$ and one can ignore its role completely in the argument below.}
 We argue that this can be achieved by the following procedure:
 let $\prec$ be the lexicographic ordering on $\mathsf{A}$ induced by $|\cdot|$, that is $\beta\prec\gamma$ if either $|\beta|<|\gamma|$ or  $|\beta|=|\gamma|$ and there is $\ell\in[d]$ with $\beta_j=\gamma_j$ for all $j\in[\ell-1]$ and $\beta_\ell< \gamma_\ell$.
 Such $\ell$ as in the last property we call the first differing component of $|\beta|=|\gamma|$.
 For $\beta=0\in \mathsf{A}$ set $\kappa_\beta:=\eps_{|\beta|}:=\rho_{1\beta}:=\ldots:=\rho_{d\beta}:=1$.
 Now suppose that $(\kappa_{\beta},\rho_{\beta})$ have been selected for all $\beta \in \mathsf{A}$ with $\beta\prec\gamma$ and $\eps_{|\beta|}$ has been chosen for all $|\beta|<|\gamma|$. 
 We now define $\kappa_{\gamma},\rho_{\gamma}$ and $\eps_{|\gamma|}$ as follows: first choose $\kappa_\gamma$ so large that $\kappa_\beta\eps_{|\beta|}^{|\gamma|-|\beta|}\rho_\beta^{\mathfrak{s}(\gamma-\beta)}\le \frac1{|\mathsf{A}|-1}\delta\kappa_\gamma$ for all $|\beta|<|\gamma|$ and $\kappa_\beta\rho_\beta^{\mathfrak{s}(\gamma-\beta)}\le \frac1{|\mathsf{A}|-1}\delta\kappa_\gamma$ for all remaining $\beta\prec\gamma$.
 If there is no $\beta\prec\gamma$ with $|\beta|=|\gamma|$, set $\rho_\gamma:=(1,\ldots,1)$.
 Otherwise let $j\in[d]$ be the largest number such that there exists $\beta\in\mathsf{A}$ with $\beta\prec\gamma$ and $|\gamma|=|\beta|$ having $j$ as the first differing component.
 Set $\rho_{(j+1)\gamma}:=\ldots:=\rho_{d\gamma}:=1$ and choose $\rho_{j\gamma}\in \N$ so large that $\kappa_{\gamma}\rho_{j\gamma}^{\mathfrak{s}_j(\beta_j-\gamma_j)}\le \frac1{|\mathsf{A}|-1}\delta\kappa_{\beta}$ for all $\beta\prec\gamma$ with $|\beta|=|\gamma|$ and first differing component being $j$.
 Then choose $\rho_{(j-1)\gamma}\in\N$ so large that $\kappa_{\gamma}\rho_{(j-1)\gamma}^{\mathfrak{s}_{j-1}(\beta_{j-1}-\gamma_{j-1})}\rho_{j\gamma}^{\mathfrak{s}_j(\beta_j-\gamma_j)}\le \frac1{|\mathsf{A}|-1}\delta\kappa_{\beta}$ for all $\beta\prec\gamma$ with $|\beta|=|\gamma|$ and first differing component being $j-1$.
 Proceed like this to define all $\rho_{\ell\gamma}\in\N$ for $\ell\in[d]$.
 Observe that by this construction we have achieved $\kappa_\gamma\rho_\gamma^{\mathfrak{s}(\beta-\gamma)}\le\frac1{|\mathsf{A}|-1}\delta\kappa_\beta$ for all $\beta\prec\gamma$ with $|\beta|=|\gamma|$.
 Having defined all $\kappa_\mu$ and $\rho_\mu$ with $|\mu|=|\gamma|$, define $\eps_{|\gamma|}\in\N$ so large that $\kappa_{\mu}\eps_{|\gamma|}^{|\beta|-|\gamma|}\rho_\mu^{\mathfrak{s}(\beta-\mu)}\le \frac1{|\mathsf{A}|-1}\delta\kappa_\beta$ for all $|\mu|=|\gamma|$ and $|\beta|<|\gamma|$.
 Since for any $\gamma\in\mathsf{A}$ the sum in \eqref{js150_hyb} contains at most $|\mathsf{A}|-1$ terms, this proves the claim.
\end{proof}

\subsection{Locally Uniform Norms}
In this section, we generalize Theorem \ref{thm:wholespace} a bit further, considering again the elliptic problem in the whole space $\R^{d}$ but restricting the active variable to a neighborhood of the base-point, and similarly restricting the magnitude of the convolution parameter.  The estimate is then shown to be uniform in the size of the neighborhood (and the convolution parameter), modulo an additional supremum norm of the germ which fades away as the size of the ball tends to infinity.  The argument is inspired in part by Simon's treatment of localized Schauder estimates in Theorem 2 of \cite{simon1997schauder}. 

\medskip

For a germ $U$ over $\R^{d}$ and $R>0$ we introduce the quantity 
\begin{equation}
\|U\|_{<R}:=
    \sup_{\substack{x,y\in\R^d\\ d(x,y)<R}}|U_{x}(y)|,
\end{equation}
and define the locally uniform semi-norms
\begin{equation}
[U]_{G_R^\gamma(\R^d)}:=
\begin{cases}
{\displaystyle \sup_{\substack{x,y\in\R^d \\0<d(x,y)<R}}  \frac{| U_{x}(y) |}{d(x,y)^{\gamma}}} &\quad \text{if} \quad \gamma>0 \\
{\displaystyle\sup_{\varphi\in \mathcal{B}}\sup_{\lambda<R}\sup_{x\in \R^d} \lambda^{-\gamma}\big |\langle U_{x},\varphi_{x}^{\lambda} \rangle \big |}   &\quad \text{if} \quad \gamma<0.
\end{cases}
\end{equation}
Given $\gamma>0$ and $\alpha \in (0,\gamma)$ we also define the quantity
\begin{equation}
[U]_{G_R^{\gamma,\alpha}(\R^d)}:= \sup_{\substack{x,y\in\R^d \\ 0<d(x,y)<R}}\inf_{P \in \mathcal{P}_{\lfloor \gamma \rfloor}} \sup_{\substack{z\in \R^d \\ 0<d(y,z)<R}} \frac{\big | (U_{x}-U_{y}-P)(z) \big | }{d(y,z)^{\alpha} \big ( d(x,y)+ d(y,z)\big )^{\gamma-\alpha}}\nonumber.
\end{equation}
For the rest of the section, we will suppress the domain $\R^d$ in the notation of the semi-norms.
We have the following variant of Lemma \ref{lemma:Rescaling}.
\begin{lemma}\label{lemma:rescaling2} 
Let $R>0$ and $w \in \R^{d}$.  Suppose $\mathcal{L}$ is of anisotropic order $m$, and let $\eta\in(0,m)$, $\alpha\in (0,\eta)$.
For all germs $U$ over $\R^{d}$ is holds
\begin{equation}
[S_{w}^R U]_{G_1^\eta }=R^\eta [U]_{G_R^\eta }, \quad \quad   [S_{w}^R U]_{G_1^{\eta,\alpha}} =R^\eta [U]_{G_R^{\eta,\alpha}} \quad \quad   [\mathcal{L} S_{w}^R U]_{G_1^{\eta-m}}=R^\eta[\mathcal{L}U]_{G_R^{\eta-m}}.  \label{G5_2}
\end{equation}
\end{lemma}
\begin{proof}
Follows the same way as the proof of Lemma \ref{lemma:Rescaling}.
\end{proof}
 \begin{theorem}\label{thm:wholespace2}
Let $m\in\N$ and let $\mathcal{L}$ be an elliptic operator of anisotropic order $m$.
Suppose $0<\alpha<\eta<m$ and $\alpha,\eta \notin \N$.
There exists a constant $C$ such that for all $\rho>0$ and all germs $U$ on $\R^{d}$ with $[U]_{G_\rho^\eta(\R^d)}<\infty$, it holds
\begin{equation}
[U]_{G_\rho^\eta} \leq C \big (  [\mathcal{L} U]_{G_\rho^{\eta-m}}+[U]_{G_\rho^{\eta,\alpha}} + \rho^{-\eta}\|U\|_{<\rho}\big ). \label{G1_2}
\end{equation}
\end{theorem}
\begin{proof}
By a scaling argument and Lemma \ref{lemma:rescaling2}, we may assume that $\rho=1$.
The proof is based on two general steps: we start with an argument based on a similar skeleton as for the proof of Theorem \ref{thm:continuum}, but with a slightly modified statement.  Namely, we argue that there are $C>0$, $\delta \in (0,1)$ such that 
\begin{align}\label{js010}
    [U]_{G_{\delta}^\eta}\le \frac12 [U]_{G_1^\eta} + C([\mathcal{L}U]_{G_1^{\eta-m}}+[U]_{G_1^{\eta,\alpha}}).
\end{align}
In a second step, we quickly deduce the main estimate \eqref{G1_2} from \eqref{js010}.  We now give the argument for \eqref{js010}.  If the claim were false, then for all $C,\delta>0$ there would exist a germ $U_{C,\delta}$ such that
\begin{align*}
    [U_{C,\delta}]_{G_{\delta}^\eta}>\frac12 [U_{C,\delta}]_{G_{1}^\eta} + C([\mathcal{L}U_{C,\delta}]_{G_{1}^{\eta-m}} + [U_{C,\delta}]_{G_{1}^{\eta,\alpha}}).
\end{align*}
In particular, choosing $C:=n$ and $\delta:=\frac1n$, we have for all $n\in \N$ a germ $\hat U^n$ such that
\begin{align*}
    [\hat U^{n}]_{G_{1/n}^\eta}>\frac12 [\hat U^{n}]_{G_{1}^\eta} + n([\mathcal{L}\hat U^{n}]_{G_{1}^{\eta-m}} + [\hat U^{n}]_{G_{1}^{\eta,\alpha}}).
\end{align*}
Introducing the germs $\tilde U^n:=n^{\eta}S_0^{1/n}\hat U^n$, we obtain from Lemma \ref{lemma:rescaling2}
\begin{align*}
    [\tilde U^{n}]_{G_1^\eta}>\frac12 [\tilde U^{n}]_{G_n^\eta} + n([\mathcal{L}\tilde U^{n}]_{G_n^{\eta-m}} + [\tilde U^{n}]_{G_n^{\eta,\alpha}}).
\end{align*}
Normalizing, we obtain
a sequence of germs $\tilde{U}^{n}$ with
\begin{align*}
[\tilde{U}^{n}]_{G_1^\eta}=1, \quad \frac12 [\tilde{U}^{n}]_{G_n^\eta} + n([\mathcal{L} \tilde{U}^{n}]_{G_n^{\eta-m}}+[\tilde{U}^{n}]_{G_n^{\eta,\alpha}}) \leq 1,  
\end{align*}
and as in the proof of Theorem \ref{thm:continuum} we find sequences $x_{n},y_{n}$ with $R_{n}:=d(x_{n},y_{n})\le 1$ such that
\begin{equation}
\big |\tilde{U}^{n}_{x_{n}}(y_{n})\big | \geq \frac{1}{2}R_{n}^{\eta}\label{G4_2}.
\end{equation}
We `blow up' near $x_{n}$ by defining a new sequence of germs $U^{n}$ via
\begin{equation}
U^{n}_{x}:=R_{n}^{-\eta}\tilde{U}_{S_{x_{n}}^{R_{n}}(x)} \circ S_{x_{n}}^{R_{n}}.\nonumber
\end{equation}
It follows by \eqref{G5_2} of Lemma \ref{lemma:rescaling2} that
\begin{equation}
[U^{n}]_{G_{1/R_n}^\eta}=1, \quad \frac12 [U^{n}]_{G_{n/R_n}^\eta} + n([\mathcal{L} U^{n}]_{G_{n/R_n}^{\eta-m}}+[U^{n}]_{G_{n/R_n}^{\eta,\alpha}}) \le 1  \label{js003_2}.
\end{equation}
Furthermore, defining $z_{n}:=(S_{x_{n}}^{R_n})^{-1}(y_n)$ (these are points on the anisotropic unit sphere), the lower bound \eqref{G4_2} turns into
\begin{equation}
\big |U^{n}_{0}(z_{n})\big | \geq \frac{1}{2} \label{G11_2}.
\end{equation}
We now define a sequence of continuous functions $u^{n}: \R^{d} \mapsto \R$ by
\begin{equation}
u^{n}(y):=U^{n}_{0}(y)\label{G12_2}.
\end{equation}

\medskip

We argue that $u^{n}$ is \textit{locally uniformly bounded} and \textit{equicontinuous}.  Indeed, by \eqref{js003_2}, there holds for all $n\in\N$, $R\in (0,n/R_n)$ and $y\in B_R(0)$
\begin{align*}
    |u^n(y)|&=|U_0^n(y)|\le [U^n]_{G_{n/R_n}^\eta} R^\eta\le 2 R^\eta,
\end{align*}
that is
\begin{align}
 \| u^{n}\|_{B_R} \lesssim R^\eta.  \label{E1_2}
\end{align}
Since $n/R_n\to \infty$ due to $R_n\in (0,1]$, this shows that $u^n$ is locally uniformly bounded.
We may also estimate the local $C^{\alpha}$ norm of $u^{n}$ by a straight forward adaptation of Lemma \ref{lemma:Holderbound_gen_hyb}, yielding for $n\in\N$ and $R\in (0,n/R_n)$ that
\begin{align*}
    \sup_{x\in\R^d} [U_x^n]_{C^\alpha(B_R(x))}\le ([U^n]_{G_{n/R_n}^\eta}+[U]_{G_{n/R_n}^{\eta,\alpha}})R^{\eta-\alpha}.
\end{align*}
Hence, $u^n$ is locally equicontinuous.
By Arzel\`a-Ascoli's theorem, passing to a subsequence (which we omit), the sequence $u^{n}$ converges locally uniformly to a continuous, locally bounded limit $u$.
We next argue that $u$ is $\mathcal{L}$-harmonic by a modification of the argument for Lemma \ref{lemma:Holderbound3}. 
Indeed, note that for $\lambda\in (0,1)$
\begin{equation}
\big |\langle \mathcal{L} u^{n}, \varphi_{y}^{\lambda} \rangle \big |
=\big |\langle \mathcal{L} U^{n}_{0}, \varphi_{y}^{\lambda} \rangle \big |\leq \big |\langle \mathcal{L} U^{n}_{y}, \varphi_{y}^{\lambda} \rangle \big |+|\langle \mathcal{L} (U^{n}_{0}-U^{n}_{y}) , \varphi_{y}^{\lambda} \rangle \big | \nonumber.
\end{equation}
Observe that for each fixed $y \in \R^{d}$ and $\lambda\in (0,1)$,  both terms on the RHS tend to zero as $n \to \infty$ by \eqref{js003_2}.
Indeed, for the first term, as long as $\lambda <\frac{n}{R_{n}}$, which is true for $n$ large enough, the quantity is bounded by $\frac{1}{n}\lambda^{\eta-m}$ using the definition combined with the second point in \eqref{js003_2}, and $n/R_n\to \infty$.
For the second term, let $n_0\in\N$ be so large that $y\in B_{n/R_n}(0)$ and $\supp\varphi_y^\lambda \subseteq y + B_{n/R_n}(y)$ for all $n\ge n_0$. Then we have
\begin{align}
&|\langle \mathcal{L} (U^{n}_{0}-U^{n}_{y} ) , \varphi_{y}^{\lambda} \rangle \big | \nonumber \\
&=|\langle \mathcal{L} (U^{n}_{0}-U^{n}_{y} - P^n ) , \varphi_{y}^{\lambda} \rangle \big | \nonumber \\
&\le\int \big |(U^{n}_{0}-U^{n}_{y} - P^n)(z)\big | |\mathcal{L}^* \varphi_{y}^{\lambda}(z)|dz  \nonumber \\
&\le [U^{n}]_{G_{n/R_n}^{\eta,\alpha}} \int d(y,z)^{\alpha} \big ( d(y,z)+d(0,y) \big )^{\eta-\alpha}|\mathcal{L}^* \varphi_{y}^{\lambda}(z)|dz  \stackrel{n\to\infty}{\longrightarrow} 0 \nonumber.
\end{align} 
In summary, we find for each $y\in \R^d$ and $\lambda\in(0,1)$
\begin{align*}
    \langle \mathcal{L} u, \varphi_{y}^{\lambda} \rangle = \langle u,\mathcal{L}^*\varphi_y^\lambda\rangle=\lim_{n\to\infty}\langle u^n,\mathcal{L}^*\varphi_y^\lambda\rangle = \lim_{n\to\infty}\langle \mathcal{L} u^n,\varphi_y^\lambda\rangle=0.
\end{align*}
 Letting $\lambda \to 0$, this implies that $u$ is $\mathcal{L}$-harmonic as a distribution (see Theorem 4.1.4 in \cite{Hor83}) and thus by ellipticity a smooth function.  As in the proof of Theorem \ref{thm:wholespace} this yields a contradiction.  This proves \eqref{js010}.

\medskip

We now turn to the second step in the proof and claim that \eqref{js010} implies that there is $C>0$ such that we have
 \begin{align}\label{js011}
     [U]_{G_1^\eta}\le \frac12 [U]_{G_1^\eta} + C([\mathcal{L}U]_{G_1^{\eta-m}}+[U]_{G_1^{\eta,\alpha}}+ \|U\|_{<1}),
 \end{align}
 which in turn immediately implies \eqref{G1_2} by absorption (and the fact that we may restrict to $\rho=1$).
 Indeed, for all $x,y\in \R^d$ with $0<d(x,y)<\delta$ we have
 \begin{align*}
     \frac{|U_x(y)|}{d(x,y)^\eta}&\le [U]_{G_{\delta}^\eta}\le \frac12 [U]_{G_1^\eta} + C([\mathcal{L}U]_{G_1^{\eta-m}}+[U]_{G_1^{\eta,\alpha}}),
 \end{align*}
 while for $x,y\in \R^d$ with $\delta\le d(x,y)< 1$ we have
 \begin{align*}
     \frac{|U_x(y)|}{d(x,y)^\eta}\le \delta^{-\eta}\|U\|_{<1}.
 \end{align*}
 Combining these two estimates gives \eqref{js011}.
\end{proof}

\subsection{Discrete Elliptic Operators}\label{sec:discrete_el}
We recall the definition of $|\gamma|$ in \eqref{e51} given a scaling $\mathfrak{s}\in \N^d$.
In this section, we derive a priori estimates for discrete elliptic operators on $\Lambda_\epsilon:=\epsilon^{\mathfrak{s}_1}\Z\times\cdots\times \epsilon^{\mathfrak{s}_d}\Z$, $\epsilon>0$,\footnote{We advise the reader to set $\epsilon=1$ at a first reading.} of the form
\begin{align*}
     \mathcal{L}_\epsilon=\sum_{|\gamma|+|\delta|= m}a_{\gamma,\delta}D_\epsilon^\gamma\overline{D}{}_\epsilon^\delta,
\end{align*}
where for $\gamma\in\N_0^d$ we write $D_\epsilon^\gamma u=D_{\epsilon,1}^{\gamma_1}\cdots D_{\epsilon,d}^{\gamma_d} u$ with the forward and backward difference operators
\begin{align*}
    D_{\epsilon,j}:=\epsilon^{-\mathfrak{s}_j}(\tau_{\epsilon,j}-\id), \quad \text{resp.} \quad \overline{D}_{\epsilon,j}:=\epsilon^{-\mathfrak{s}_j}(\id-\tau_{\epsilon,j}^{-1}),
\end{align*}
and the translation operators $\tau_{\epsilon,j}$ defined by $\tau_{\epsilon,j}u(k):=u(k+\epsilon^{\mathfrak{s}_j} \mathbf{e}_j)$.
Observe that for $u\in C^\infty(\R^d)$ it holds $\lim_{\epsilon\searrow 0} D_{\eps,j}u=\lim_{\epsilon\searrow 0} \overline{D}_{\eps,j}u=\partial_j u$, and hence $\lim_{\epsilon\searrow 0} \mathcal{L}_\epsilon u=\mathcal{L} u$ with $\mathcal{L}:=\sum_{|\gamma|+|\delta|=m} a_{\gamma,\delta} \partial^{\gamma+\delta}$.
Let us write $\widehat{\Lambda}_\epsilon:=\mathbb{T}_{\epsilon^{-\mathfrak{s}_1}}\times \cdots \times \mathbb{T}_{\epsilon^{-\mathfrak{s}_d}}$ with $\mathbb{T}_r:=\R/2\pi r\Z$  for the Pontryagin dual group of $\Lambda_\epsilon$, see Chapter 1 in \cite{Rud62}.\footnote{We often identify $\mathbb{T}_r$ with the interval $[-\pi r,\pi r)$ if no confusion can arise.}
Moreover, we write $\xi_{\theta,j}^\epsilon:=i\epsilon^{-\mathfrak{s}_j}(1-e^{i\epsilon^{\mathfrak{s}_j}\theta_j})$ and use for $\gamma\in\N_0^d$ the notation $(\xi_\theta^\epsilon)^\gamma=\prod_{j=1}^d (\xi_{z,j}^\epsilon)^{\gamma_j}$.
Following \cite{ThW68}, we call $\mathcal{L}_\epsilon$ discretely elliptic on $\Lambda_\epsilon$ if
\begin{align}\label{js1011}
 \hat{\mathcal{L}_\epsilon}(\theta):=\sum_{|\gamma|+|\delta|= m}a_{\gamma,\delta} (i\xi_\theta^\epsilon)^\gamma (i\overline{\xi_\theta^\epsilon})^\delta\ne 0 \  \text{ if } \ \theta\in\widehat{\Lambda}_\epsilon\setminus\{\mathsf{0}\}   
\end{align}
and 
\begin{align*}
 \hat{\mathcal{L}}(\xi):=\sum_{|\gamma|+|\delta|= m}a_{\gamma,\delta} (i\xi)^{\gamma+\delta} \ne 0 \ \text{ for all } \ \xi\in\mathbb{R}^d\setminus\{0\}.   
\end{align*}
We note that \eqref{js1011} is natural in view of the Liouville theorem, see Lemma \ref{js1000} in the appendix.
Let us also mention that we may view $\xi\in \R^d$ as an element of $\widehat{\Lambda}_\eps$ for sufficiently small $\eps>0$, and it holds
\begin{align*}
\lim_{\epsilon\searrow 0} i\xi_{\xi,j}^\epsilon = \lim_{\epsilon\searrow 0} i\overline{\xi_{\xi,j}^\epsilon} = i\xi_j.
\end{align*}
In particular $\hat{\mathcal{L}}(\xi)=\lim_{\epsilon\searrow 0}\hat{\mathcal{L}_\epsilon}(\xi)$ for all $\xi\in\R^d$.
\begin{example}\label{js1010}
Let $\mathfrak{s}:=(1,\ldots,1)$ and $\epsilon>0$.
\begin{enumerate}
    \item\label{js1010i} Write $|\xi_\theta^\epsilon|^2:=\sum_{j=1}^d |\xi_{\theta,j}^\epsilon|^2$ for $\theta\in \widehat{\Lambda}_\epsilon$.\footnote{Observe that $|\xi_\theta^\epsilon|^2=\epsilon^{-2}\sum_{j=1}^d |e^{i\epsilon\theta_j}-1|^2 = 4\epsilon^{-2}\sum_{j=1}^d \sin^2\frac{\epsilon\theta_j}{2}$ for $\theta\in \widehat{\Lambda}_\epsilon$, i.e.~$\theta_j\in [-\frac{\pi}{\epsilon},\frac{\pi}{\epsilon})$ for all $j\in[d]$.}
    Then for the discrete Laplacian $\mathcal{L}_\epsilon=\Delta_\epsilon= \sum_{j=1}^d D_{\epsilon,j}\overline{D}_{\epsilon,j}$ we have $\hat{\mathcal{L}_\epsilon}(\theta)=-|\xi_\theta^\epsilon|^2$ and $\hat{\mathcal{L}}(\xi)=-|\xi|^2$.
    Since $|\xi_\theta^\epsilon|=0$ if and only if $\theta_j=0$ for all $j\in[d]$, we see that $\Delta_\epsilon$ is discretely elliptic on $\Lambda_\epsilon$.
    \item\label{js1010ii} Consider $\mathcal{L}_\epsilon=\sum_{j=1}^d (D_{\epsilon,j}-\overline{D}_{\epsilon,j})$.
    Observe that $\mathcal{L}_\epsilon=\epsilon\Delta_\epsilon$.
    Thus $\hat{\mathcal{L}}_\epsilon(\theta)=-\epsilon|\xi_\theta^\epsilon|^2\ne 0$
    for all $\theta\in\widehat{\Lambda}_\epsilon\setminus\{0\}$, but $\mathcal{L}(\xi)=\sum_{j=1}^d (i\xi_j-i\xi_j)=0$ for all $\xi\in \R^d$.
    Therefore $\mathcal{L}_\eps$ is not discretely elliptic on $\Lambda_\epsilon$.
\end{enumerate}
\end{example}
The preceding example shows that if not explicitly assumed, ellipticity may fail in the limit $\epsilon=0$.
In contrast, the discrete property \eqref{js1011} is independent of $\epsilon>0$ for scaling homogeneous operators as shown in the next lemma. 
\begin{lemma}\label{js1012}
 Let $m\in\N$ and let $(a_{\gamma,\delta})_{\gamma,\delta}$ be a family of complex numbers.
 If \eqref{js1011} holds for some $\epsilon_0>0$, then \eqref{js1011} holds for all $\epsilon>0$.
\end{lemma}
\begin{proof}
Let $\epsilon_0>0$ be such that \eqref{js1011} is fulfilled, and let $\epsilon>0$.
Let $\theta\in \widehat{\Lambda}_\epsilon\setminus\{0\}$ and observe that $\phi:=\big(\frac{\epsilon}{\epsilon_0}\big)^{\mathfrak{s}}\theta\in \widehat{\Lambda}_{\epsilon_0}\setminus\{0\}$, where $R^\mathfrak{s}\theta:=(R^{\mathfrak{s}_1}\theta_1,\ldots,R^{\mathfrak{s}_d}\theta_d)$ for $R>0$.
The key point is to observe that for $\gamma,\delta\in\N_0^d$ with $|\gamma|+|\delta|=m$ we have $\prod_{j=1}^d \epsilon_0^{-\mathfrak{s}_j\gamma_j}\prod_{j=1}^d \epsilon_0^{-\mathfrak{s}_j\delta_j}=\epsilon_0^{-\mathfrak{s}\cdot (\gamma+\delta)}=\epsilon_0^{-m}$ by \eqref{e51}.
Therefore the definition of $\xi_{\theta,j}^{\epsilon_0}$ gives the desired
\begin{align*}
0\ne \hat{\mathcal{L}}_{\epsilon_0}(\phi)&=\epsilon_0^{-m}\sum_{|\gamma|+|\delta|=m}a_{\gamma,\delta}\prod_{j=1}^d (e^{i\epsilon_0^{\mathfrak{s}_j}\phi_j}-1)^{\gamma_j}\prod_{j=1}^d (1-e^{-i\epsilon_0^{\mathfrak{s}_j}\phi_j})^{\delta_j}=(\epsilon_0/\epsilon)^{-m} \hat{\mathcal{L}}_\epsilon(\theta).
\qedhere
\end{align*}
\end{proof}
The analogue of the negative $G^{\gamma}$-norm in \eqref{e24} is now given for a germ $V$ over $\Lambda_\epsilon$ via
\begin{align*}
[V]_{G^{\gamma}(\Lambda_\epsilon)}:=\sup_{\varphi \in \mathcal{B}}\sup_{x \in \Lambda_\epsilon,\lambda \geq \epsilon} \lambda^{-\gamma}\big |\langle V_{x},\varphi_{x}^{\lambda} \rangle_{\epsilon} \big |,
\end{align*}
where $\langle \cdot, \cdot \rangle_{\epsilon}$ is given by
\begin{align*}
    \langle f,g\rangle_{\epsilon}:=\epsilon^{\sum_{j=1}^d \mathfrak{s}_j}\sum_{x \in \Lambda_\epsilon}f(x) g(x).
\end{align*}
For $\eta>0$ we say that a germ $U$ on $\Lambda_\epsilon$ is centered to order $\eta$ provided that $D_\epsilon^{\gamma}U_{x}(x)=0$ for all $x\in \Lambda_\epsilon$ and $|\gamma|\le \eta$.
\begin{lemma}\label{js1001}
    Let $\eta,\epsilon>0$. If $u:\Lambda_\epsilon\to \C$ is the restriction of a polynomial to $\Lambda_\epsilon$ of anisotropic order at most $\eta$ such that $D_\epsilon^\gamma u(0)=0$ for all $|\gamma|\le \eta$, then $u=0$.
\end{lemma}
\begin{proof}
    We may write $u(k)=\sum_{|\gamma|\le \eta} c_\gamma k_\epsilon^{(\gamma)}$,
    where the discrete monomials are defined via
    \begin{align*}
     k_\epsilon^{(\gamma)}:=\prod_{j=1}^d k_{\epsilon,j}^{(\gamma_j)}, \quad \quad \text{and} \quad  k_{\epsilon,j}^{(\gamma_j)}:=\prod_{m=0}^{\gamma_j-1} (k_j-\epsilon^{\mathfrak{s}_j}m),
    \end{align*}
    (cf.~Chapter 3.3.2 in \cite{RuT10}).
    By induction it follows $D_\epsilon^\gamma k_\epsilon^{(\delta)}=\frac{\delta!}{(\delta-\gamma)!} k_\epsilon^{(\delta-\gamma)}$ if $\gamma\le \delta$ and $D_\epsilon^\gamma k_\epsilon^{(\delta)}=0$ else.
    Note that for $\gamma\ne \delta$ the expression $\frac{\delta!}{(\delta-\gamma)!} k_\epsilon^{(\delta-\gamma)}$ vanishes in the origin $k=0$.
    In particular $0=D_\epsilon^\gamma u(0)= \gamma^{(\gamma)} c_\gamma$ for $|\gamma|\le \eta$, i.e.~$u=0$.
\end{proof}
The Schauder estimate for germs is now stated as follows.
\begin{theorem}\label{thm:discrete_gen}
Let $m\in\N$ and let $(a_{\gamma,\delta})_{|\gamma|+|\delta|=m}$ be a family of complex numbers such that for all $\epsilon>0$ the discrete difference operator $\mathcal{L}_\epsilon:=\sum_{|\gamma|+|\delta|= m}a_{\gamma,\delta}D_\epsilon^\gamma\overline{D}{}_\epsilon^\delta$ is elliptic on $\Lambda_\epsilon$.\footnote{By Lemma \ref{js1012}, it suffices to check this for \emph{one} $\epsilon>0$.}
 Suppose $0<\alpha<\eta<m$ with $\alpha,\eta\notin \N$.
 There exists a constant $C>0$ such that for all $\epsilon>0$ and germs $U$ centered to order $\eta$ on $\Lambda_\epsilon$ with $\|U\|_{G^{\eta}(\Lambda_\epsilon)}<\infty$ it holds 
\begin{equation}
    \|U\|_{G^{\eta}(\Lambda_\epsilon)} \leq C \big (  [\mathcal{L}_\epsilon U]_{G^{\eta-m}(\Lambda_\epsilon)}+[U]_{G^{\eta,\alpha}(\Lambda_\epsilon)} \big ) \label{js400}.
\end{equation}
\end{theorem}
\begin{proof}
The proof is similar to that of Theorem \ref{thm:discrete}, with the following modifications.
Replace $\Delta_\epsilon$ by $\mathcal{L}_\epsilon$ as well as $\Delta$ by $\mathcal{L}$, and write $m$ instead of $2$ for the anisotropic order of this operator.
As in the proof of Theorem \ref{thm:wholespace} we replace Lemma \ref{lemma:Holderbound2} by the more general statement in Lemma \ref{lemma:Holderbound_gen_hyb} to obtain a uniform H\"older bound.
In the case $\epsilon^\infty=0$ the assertion that $\langle \Delta_{\epsilon^{n}}\hat{u}^{n},\varphi_{y^{n}}^{\lambda^{n}} \rangle_{\epsilon^{n}}$ converges to $\langle \Delta u, \varphi_{y}^{\lambda} \rangle$ now reads that $\langle \mathcal{L}_{\epsilon^{n}}\hat{u}^{n},\varphi_{y^{n}}^{\lambda^{n}} \rangle_{\epsilon^{n}}$ converges to $\langle \mathcal{L} u, \varphi_{y}^{\lambda} \rangle$, which again follows by dominated convergence and the fact that both the forward and the backward difference operator approximate the continuous derivative.
Observe that $\hat{\mathcal{L}}(\xi)\ne 0$ for all $\xi\in \R^d\setminus\{0\}$ by assumption.
Therefore we may replace the reference to the classical continuous Liouville theorem by Lemma \ref{js100} in the appendix and conclude in the case $\epsilon^\infty=0$.

\medskip

In the case $\epsilon^\infty>0$, Lemma \ref{js1000} in the appendix is applicable with $\epsilon^\infty>0$ by assumption.
Together with the definition of a germ centered to order $\eta$ this shows that $u$ is identically zero.
\end{proof}

\section{Acknowledgments}
The authors thank the referees for carefully reading the manuscript and providing valuable feedback. S.S. is grateful for the financial support from National Key R\&D Program of China (No. 2022YFA1006300).

\appendix

\section{Liouville-type Theorems}
The proofs of our main theorems rely on generalized versions of the Liouville theorem.
For completeness, we state and prove them here both in the continuous and discrete setting.
The idea for the proofs are inspired by \cite[Theorem 1]{Wec83}.
We start with the continuous setting.
Let us write $\hat u:=\mathcal{F}u$ for the Fourier transform of a tempered distribution $u\in\mathcal{S}'(\R^d)$, and $\check{u}:=\mathcal{F}^{-1} u$ for its inverse Fourier transform.
The Fourier transform is an isomorphism on $\mathcal{S}'(\R^d)$.
If $a\in C^\infty(\R^d)$ is such that all derivatives are polynomially bounded, we write $a\in \mathcal{O}_M (\R^d)$.
For $T\in \mathcal{S}'(\R^d)$ and $a\in \mathcal{O}_M (\R^d)$ we define the multiplication $aT\in \mathcal{S}'(\R^d)$ via $\langle aT,\varphi\rangle:=\langle T,a\varphi\rangle$ for all $\varphi\in \mathcal{S}(\R^d)$, see Theorem 7.10 in \cite{Sch66}.
\begin{lemma}\label{js100}
 Let $\hat{\mathcal{L}}\in \mathcal{O}_M (\R^d)$.
 For $u\in \mathcal{S}'(\R^d)$ define $\mathcal{L}u:=\mathcal{F}^{-1}[\hat{\mathcal{L}}\hat{u}]\in \mathcal{S}'(\R^d)$.
 \begin{enumerate}
     \item\label{js100i} $\hat{\mathcal{L}}(\xi)
     \ne 0$ for all $\xi\in\R^d\setminus\{0\}$.
     \item\label{js100ii} If $u\in \mathcal{S}'(\R^d)$ satisfies $\mathcal{L}u=0$, then $u$ is a polynomial.
     \item\label{js100iii} For all $\eta\ge 0$ and $u\in L^{1}_{\mathrm{loc}}(\R^d)$ with $\mathcal{L}u=0$ and $\|u\|_{B_R(0)}\le CR^\eta$ for some $C\ge 0$ and all $R\ge 1$, it holds that $u$ is polynomial of anisotropic order at most $\eta$.
     \item\label{js100iv} $u(x):=e^{ix\cdot\xi}$ satisfies $\mathcal{L}u=0$ only if $\xi=0$.
 \end{enumerate}
 Moreover $\hat{\mathcal{L}}(\xi)\ne 0$ for all $\xi\in \R^d$ if and only if $u=0$ whenever $u\in \mathcal{S}'(\R^d)$ satisfies $\mathcal{L}u=0$.
\end{lemma}
\begin{proof}
 \ref{js100i}$\Rightarrow$\ref{js100ii}:
 Let $\varphi\in C^\infty_c(\R^d\setminus\{0\})$.
 Then $\psi:=\hat{\mathcal{L}}^{-1}\varphi\in \mathcal{S}(\R^d)$, and thus $\langle\hat u,\varphi\rangle=\langle\hat u,\hat{\mathcal{L}}\psi\rangle=\langle\hat{\mathcal{L}}\hat{u},\psi\rangle=\langle \mathcal{L}u,\hat\psi\rangle=0$.
 This shows $\supp\hat{u}\subset\{0\}$, which by Theorem 6.25 in \cite{Rud91} means that $\hat u$ is a finite linear combination of derivatives of the $\delta$-distribution, i.e., $u$ is a polynomial.
 If $\hat{\mathcal{L}}(\xi)\ne 0$ for all $\xi\in \R^d$, then the same argument with $\varphi\in C^\infty_c(\R^d)$ shows that $u=0$.

 \medskip

 \ref{js100ii}$\Rightarrow$\ref{js100iii}: Any polynomially bounded $u\in L^1_{\mathrm{loc}}(\R^d)$ is a tempered distribution, and the anisotropic order of a polynomial with $\|u\|_{B_R(0)}\le CR^\eta$ for all $R\ge 1$ is at most $\eta$.

 \medskip

 \ref{js100iii}$\Rightarrow$\ref{js100iv}: Follows by taking $\eta=0$, since $u(x)=e^{i x\cdot \xi}$ is non-constant and bounded for $\xi\in\R^d\setminus\{0\}$.

 \medskip
 
 \ref{js100iv}$\Rightarrow$\ref{js100i}: If $\xi\in \R^d$ fulfills $\hat{\mathcal{L}}(\xi)=0$, then $u(x):=e^{i x\cdot \xi}$ solves $\mathcal{L}u=0$.
 Indeed, since $\hat{u}=\delta_{\xi}$, we have $\langle\mathcal{L}u,\varphi\rangle=\langle\hat{\mathcal{L}}\hat{u},\check{\varphi}\rangle=\langle\hat{u},\hat{\mathcal{L}}\check{\varphi}\rangle=\hat{\mathcal{L}}(\xi)\check{\varphi}(\xi)=0$ for any $\varphi\in \mathcal{S}(\R^d)$.
\end{proof}

Operators of above type are called pseudodifferential operators with constant coefficients.
By choosing $\hat{\mathcal{L}}$ to be a polynomial, we see that the above lemma applies in particular to differential operators on $\R^d$ with constant coefficients.
The choice $\hat{\mathcal{L}}(\xi)=|\xi|^{s}$ with $s>0$ however, is not immediately included in the above lemma unless $s$ is an even integer, even not for the isotropic case $\mathfrak{s}=(1,\ldots,1)$.
Thus the corresponding operator, the fractional Laplacian $\mathcal{L}:=(-\Delta)^s$, is defined only on a subspace of $\mathcal{S}'(\R^d)$, for example on $L^1_s(\R^d):=\{u\in L^1(\R^d): \|\langle\cdot\rangle^{-(s+d)}u\|_{L^1(\R^d)}<\infty\}$ with $\langle x\rangle:=(1+|x|^2)^{\frac12}$ via $\langle \mathcal{L} u,\varphi\rangle:=\langle u,\mathcal{F}|\cdot|^s\check\varphi\rangle$ for all $\varphi\in\mathcal{S}(\R^d)$.
The following lemma is a slight generalization of Theorem 1.3 in \cite{CDL15}.
\begin{lemma}
Let $\mathfrak{s}=(1,\ldots,1)$ and $s>0$.
Then any $u\in L^1_s(\R^d)$ with $(-\Delta)^s u=0$ is a polynomial of order at most $s$.
\end{lemma}
\begin{proof}
 Let $\varphi\in C^\infty_c(\R^d\setminus\{0\})$ and write $\hat{\mathcal{L}}(\xi):=|\xi|^s$.
 Then $\psi:=\mathcal{F}\hat{\mathcal{L}}^{-1}\varphi\in \mathcal{S}(\R^d)$.
 Since $L^1_s(\R^d)\subset \mathcal{S}'(\R^d)$, it follows $\langle\hat u,\varphi\rangle=\langle\hat u,\hat{\mathcal{L}}\check\psi\rangle=\langle u,\mathcal{L}\psi\rangle=\langle\mathcal{L}u,\psi\rangle=0$.
 This shows $\supp\hat{u}\subset\{0\}$, which by Theorem 6.25 in \cite{Rud91} means that $\hat u$ is a finite linear combination of derivatives of the $\delta$-distribution, i.e., $u$ is a polynomial.
 Since the order of any polynomial in $L^1_s(\R^d)$ is at most $s$, this gives the result.
\end{proof}

\medskip

For the discrete case let $\eps>0$ and recall that a tempered distribution on $\Lambda_\epsilon:=\epsilon^{\mathfrak{s}_1}\Z\times\cdots\times \epsilon^{\mathfrak{s}_d}\Z$
is simply a function $u:\Lambda_\epsilon\to\C$ which has at most polynomial growth, see Chapter 7.1 in \cite{Sch66}.
We denote the space of tempered distributions on $\Lambda_\epsilon$ by $\mathcal{S}'(\Lambda_\epsilon)$,
and define the (invertible) discrete Fourier transform\footnote{Note that the Schwartz-Bruhat space $\mathcal{S}(\widehat{\Lambda}_\epsilon)$ of the Pontryagin dual group $\widehat{\Lambda}_\epsilon$ is given by $C^\infty(\widehat{\Lambda}_\epsilon)$, see \cite{Bru61}, and consequently the concept of tempered distributions coincides with the concept of distributions on $\widehat{\Lambda}_\epsilon$.}\\
\begin{minipage}{.1\textwidth}
\phantom{aaaa}
\end{minipage}
\begin{minipage}{.4\textwidth}
\begin{align*}
    \mathcal{F}_{\Lambda_\epsilon}&: \mathcal{S}'(\Lambda_\epsilon)\to \mathcal{D}'(\widehat{\Lambda}_\epsilon),\\
    \mathcal{F}_{\Lambda_\epsilon}&u:=\hat u:=(2\pi)^{-d}\epsilon^{\sum_{j=1}^d \mathfrak{s}_j}\sum_{k\in\Lambda_\epsilon}u(k)e_{-k},
\end{align*}
\end{minipage}
\begin{minipage}{.5\textwidth}
\begin{align*}
    \mathcal{F}_{\Lambda_\epsilon}^{-1}&: \mathcal{D}'(\widehat{\Lambda}_\epsilon)\to \mathcal{S}'(\Lambda_\epsilon),\\
    \mathcal{F}_{\Lambda_\epsilon}^{-1}&T(k):=\langle T,e_{k}\rangle_{\widehat{\Lambda}_\epsilon},
\end{align*}
\end{minipage}
where $e_k\in C^\infty(\widehat{\Lambda}_\epsilon)$ is defined by $e_k(\theta):=e^{i\theta\cdot k}$.
Here we identify $v\in L^1_{\mathrm{loc}}(\widehat{\Lambda}_\epsilon)$ with $v\in \mathcal{D}'(\widehat{\Lambda}_\epsilon)$ defined by\footnote{Here $\int_{\widehat{\Lambda}_\epsilon} \varphi\,\mathrm{d}\theta:=\int_{-\pi\epsilon^{-\mathfrak{s}_d}}^{\pi\epsilon^{-\mathfrak{s}_d}}\cdots \int_{-\pi\epsilon^{-\mathfrak{s}_1}}^{\pi\epsilon^{-\mathfrak{s}_1}} \varphi\,\mathrm{d}\theta_1 \cdots\mathrm{d}\theta_d$, in particular $\int_{\widehat{\Lambda}_\epsilon} \,\mathrm{d}\theta=(2\pi)^d\epsilon^{-\sum_{j=1}^d\mathfrak{s}_j}$.} 
$\langle v,\varphi\rangle_{\widehat{\Lambda}_\epsilon}:=\int_{\widehat{\Lambda}_\epsilon}v\varphi \, \mathrm{d} \theta$, and for $a,\varphi\in C^\infty(\widehat{\Lambda}_\epsilon)$ and $T\in \mathcal{D}'(\widehat{\Lambda}_\epsilon)$ we define multiplication by $\langle aT,\varphi\rangle_{\widehat{\Lambda}_\epsilon}:=\langle T, a\varphi\rangle_{\widehat{\Lambda}_\epsilon}$.
Observe that for $\theta\in \widehat{\Lambda}_\epsilon$, the inverse Fourier transform of derivatives of the delta distribution $\delta_\theta\in \mathcal{D}'(\widehat{\Lambda}_\epsilon)$ defined via $\langle \delta_\theta,\varphi\rangle_{\widehat{\Lambda}_\epsilon}:=\varphi(\theta)$ is given by $\mathcal{F}_{\Lambda_\epsilon}^{-1}D^\gamma\delta_\theta(k)=(-ik)^\gamma e^{i\theta\cdot k}$.
Moreover, with the notation of Section \ref{sec:discrete_el} we have
\begin{align*}
 \widehat{D_\epsilon^\gamma u}=(i\xi_{\cdot}^\epsilon)^\gamma \hat u \quad \text{and} \quad \widehat{\overline{D}{}_\epsilon^\delta u}= (i\overline{\xi_{\cdot}^\epsilon})^\delta \hat u;
\end{align*}
hence the following lemma is applicable to difference operators as discussed in Section \ref{sec:discrete_el}.
\begin{lemma}\label{js1000}
 Let $\epsilon >0$ and suppose that $\hat{\mathcal{L}_\epsilon}\in C^\infty(\widehat{\Lambda}_\epsilon)$.
 For $u\in \mathcal{S}'(\Lambda_\epsilon)$ we define $\mathcal{L}_\epsilon u:=\mathcal{F}_{\Lambda_\epsilon}^{-1}[\hat{\mathcal{L}}_\epsilon \hat u]$.
 The following are equivalent:
 \begin{enumerate}
     \item\label{js1000i} $\hat{\mathcal{L}_\epsilon}(\theta)\ne 0$ for all $\theta\in\widehat{\Lambda}_\epsilon\setminus\{0\}$.
     \item\label{js1000ii} If $u\in\mathcal{S}'(\Lambda_\epsilon)$ satisfies $\mathcal{L}_\epsilon u=0$, then $u$ is a polynomial restricted to $\Lambda_\epsilon$.
     \item\label{js1000iii} For all $\eta\ge 0$ and $u:\Lambda_\epsilon\to \C$ with $\mathcal{L}_\epsilon u=0$ and $\|u\|_{B_R(0)}\le CR^\eta$ for some $C\ge 0$ and all $R\ge 1$, it holds that $u$ is a polynomial of anisotropic order at most $\eta$.
     \item\label{js1000iv} For $\theta\in \widehat{\Lambda}_\epsilon$, the function $u:\Lambda_\epsilon\to \C$ defined by $u(k):=e^{i\theta\cdot k}$ satisfies $\mathcal{L}_\epsilon u=0$ only if $\theta=0$.
 \end{enumerate}
 Moreover $\hat{\mathcal{L}}_\epsilon(\theta)\ne 0$ for all $\theta\in \widehat{\Lambda}_\epsilon$ if and only if $u=0$ whenever $u\in \mathcal{S}'(\Lambda_\epsilon)$ satisfies $\mathcal{L}_\epsilon u=0$.
\end{lemma}
\begin{proof}
 \ref{js1000i}$\Rightarrow$\ref{js1000ii}:
 Let $\varphi\in C^\infty_c(\widehat{\Lambda}_\epsilon\setminus\{0\})$.
 Then $\psi:=\varphi/\hat{\mathcal{L}}_\epsilon\in C^\infty(\widehat{\Lambda}_\epsilon)$, and thus $\langle\hat u,\varphi\rangle_{\widehat{\Lambda}_\epsilon}=\langle\hat u,\hat{\mathcal{L}}_\epsilon\psi\rangle_{\widehat{\Lambda}_\epsilon}=\langle\hat{\mathcal{L}}_\epsilon\hat{u},\psi\rangle_{\widehat{\Lambda}_\epsilon}=\langle[\mathcal{L}_\epsilon u]\widehat{\phantom{a}},\psi\rangle_{\widehat{\Lambda}_\epsilon}=0$.
 This shows $\supp\hat{u}\subset \{0\}$.
 By Chapter 7.1 in \cite{Sch66} and Theorem 6.25 in \cite{Rud91}, we find that $\hat u$ is a finite linear combination of derivatives of the $\delta_0$-distribution, and hence by Fourier inversion
 $u$ is a polynomial restricted to $\Lambda_\epsilon$.
 If $\hat{\mathcal{L}}_\epsilon(\theta)\ne 0$ for all $\theta\in \widehat{\Lambda}_\epsilon$, then the same argument with $\varphi\in C^\infty(\widehat{\Lambda}_\epsilon)$ shows that $u=0$.

 \medskip

 \ref{js1000ii}$\Rightarrow$\ref{js1000iii}: Any polynomially bounded function $u:\Lambda_\epsilon\to\C$ is a tempered distribution, and the anisotropic order of a polynomial with $\|u\|_{B_R(0)}\le CR^\eta$ for all $R\ge 1$ is at most $\eta$.

  \medskip

  \ref{js1000iii}$\Rightarrow$\ref{js1000iv}: Follows by taking $\eta=0$, since $u(k)=e^{i\theta\cdot k}$ is bounded and non-constant for $\theta\in\widehat{\Lambda}_\epsilon\setminus\{0\}$.

  \medskip
 
 \ref{js1000iv}$\Rightarrow$\ref{js1000i}: 
 If $\theta\in \widehat{\Lambda}_\epsilon$ fulfills $\hat{\mathcal{L}}_\epsilon(\theta)=0$, then $u(k):=e^{i \theta\cdot k}$ solves $\mathcal{L}_\epsilon u=0$.
 Indeed, since $\hat{u}=\delta_{\theta}$, we have for any $k\in \Lambda_\epsilon$ that $\mathcal{L}_\epsilon u(k)=\langle\hat{\mathcal{L}}_\epsilon \hat{u},e_k\rangle_{\widehat{\Lambda}_\epsilon}=\langle\hat{u},\hat{\mathcal{L}_\epsilon} e_k\rangle_{\widehat{\Lambda}_\epsilon}=\hat{\mathcal{L}_\epsilon}(\theta)e^{i\theta\cdot k}=0$.
\end{proof}

\bibliographystyle{plain}

\end{document}